\documentclass[12pt, english]{article}
\usepackage{geometry}   
\usepackage{amssymb}
\usepackage{amsmath}
\usepackage{amsfonts}
\usepackage{latexsym}
\usepackage{graphicx}
\usepackage{amsthm,enumerate,exscale,stmaryrd}
\usepackage{cases}
\usepackage{bm}
\usepackage[pdftex,bookmarks,colorlinks,breaklinks=true]{hyperref}
\usepackage{color}
\input xy
\xyoption{all}
\usepackage{dsfont}
\usepackage{yfonts}
\usepackage{bbm}
\usepackage{leftidx}
\usepackage{upgreek}
\usepackage[all]{xy}
\usepackage{verbatim}
\usepackage[multiple]{footmisc}
\usepackage{charter}
\NoComputerModernTips
\newdir^{ (}{{}*!/-3pt/\dir^{(}}
\newdir_{ (}{{}*!/-3pt/\dir_{(}}

\newcommand{\Id}{\mathop{\rm Id}\nolimits}

\newcommand{\xqed}[1]{%
  \leavevmode\unskip\penalty9999 \hbox{}\nobreak\hfill
  \quad\hbox{\ensuremath{#1}}}

\newcommand{\ovset}[2]{\overset{#1}{#2}}
\newcommand{\unset}[2]{\underset{#1}{#2}}
\def\Ab{\mathop{\rm \mathfrak{Ab}}\nolimits}
\def\Scheme{\mathop{\rm \mathfrak{Sch}}\nolimits}

\def\Alt{\mathop{\rm Alt}\nolimits}

\def\Sym{\mathop{\rm Sym}\nolimits}
\def\Mult{\mathop{\rm Mult}\nolimits}

\def\Hom{\mathop{\rm Hom}\nolimits}

\def\Spec{\mathop{\bf Spec}\nolimits}

\def\kernel{\mathop{\rm Ker}\nolimits}
\def\cokernel{\mathop{\rm Coker}\nolimits}
\def\image{\mathop{\rm Im}\nolimits}

\def\innHom{\underline{\Hom}}

\def\innMult{\underline{\Mult}}
\def\innSym{\underline{\Sym}}
\def\innAlt{\underline{\Alt}}
\let\phi\varphi
\let\epsilon\varepsilon

\newtheorem{thm}[equation]{Theorem}

\newtheorem{cor}[equation]{Corollary}

\newtheorem{lem}[equation]{Lemma}

\newtheorem{prop}[equation]{Proposition}

\theoremstyle{definition}
\newtheorem{dfn}[equation]{Definition}                
\newtheorem{ex}[equation]{Example}              
\newtheorem{rem}[equation]{Remark}

\newcounter{example}
\renewcommand{\theexample}{\arabic{example}}

\numberwithin{equation}{section}

\newcommand{\BF}{{\mathbb{F}}}
\newcommand{\BG}{{\mathbb{G}}}

\def\longto{\longrightarrow}

\def\into{\hookrightarrow}
\let\onto\twoheadrightarrow
\def\isoto{\arrover{\cong}}

\newbox\mybox
\def\arrover#1{\mathrel{
       \setbox\mybox=\hbox spread 1em
              {\hfil$\scriptstyle#1\vphantom{g}$\hfil}
       \vbox{\offinterlineskip\copy\mybox
             \hbox to\wd\mybox{\rightarrowfill}}}}             
\def\invlim{\mathop{\vtop{\hbox{\rm lim}\vskip-8pt
        \hbox{\hskip1pt$\scriptstyle\longleftarrow$}\vskip-1pt}}}

\def\ontoover#1{\mathrel{
       \setbox\mybox=\hbox spread 1.4em{\hfil$\scriptstyle#1$\hfil}
       \vbox{\offinterlineskip\copy\mybox
             \hbox to\wd\mybox{\rightarrowfill\hskip-2.8mm
                               $\rightarrow$}}}}
\def\longonto{\ontoover{\ }}
\def\to{\rightarrow}
\def\xrarr{\xrightarrow}

\title{Tensor Operations on Group Schemes}
\author{S. Mohammad Hadi Hedayatzadeh\footnote{Department of Mathematics, Institute for Research in Fundamental Sciences (IPM), Tehran, Iran \href{mailto:hadi@ipm.ir}{hadi@ipm.ir}}}

\date{}

\begin{document}

\maketitle
\abstract{In this paper we study multilinear morphisms between commutative group
schemes and the associated tensor constructions. We will also do some explicit
calculations and give examples that show that this theory behaves in
a way that one would naturally expect.}
\tableofcontents

\parindent0pt
\overfullrule=0pt

\section{Introduction}

When we are studying
homomorphisms of commutative group schemes, we are naturally led to
look at multilinear morphisms between them, because on the one hand
they are obvious generalizations of homomorphisms and on the other
hand they make it possible to have a group scheme version of
multilinear algebra. Although some of the results of multilinear
algebra are no longer valid in this new setting, there are many
similarities between these two theories, as we will see. Let
$G_1,\dots,G_n$ and $H$ be commutative group schemes over a base
scheme $S$. A multilinear morphism $\phi:G_1\times\dots\times G_n\to
H$ is a morphism of schemes over $S$ that is linear in each $G_i$.
The group of all such multilinear morphisms is denoted by
$\Mult(G_1\times\dots\times G_n,H)$. Natural examples of multilinear
morphisms are the Weil pairings on the torsion groups of an abelian
variety, and the pairing $G\times G^*\to\BG_m$ between a
finite and flat commutative group scheme $G$ and its Cartier dual $G^*$.\\

In the first section we study groups of multilinear morphisms and
related concepts. We define the so-called inner $\Hom$ of two
commutative group schemes $G$ and $H$, denoted by $\innHom(G,H)$, as
being the group scheme representing the functor
$$\Scheme_S\longto\Ab,\qquad T\mapsto \Hom(G_T,H_T)$$
from the category of schemes over $S$ to the category of abelian
groups. Theorem 3.10 in \cite{Pink1} states that this group scheme
exists whenever the group $G$ is finite and flat over $S$ and is
affine (or of finite type) if $H$ is affine (or of finite type). We
show that this construction commutes with the base change, i.e.,
$\innHom(G,H)_T\cong\innHom(G_T,H_T)$ for any $S$-scheme $T$ and
that the functors $\innHom(-,H)$ and $\innHom(G,-)$ from the
category of commutative group schemes over a base field to itself
are left exact, which is not very surprising, since these functors
are constructed from left exact functors $\Hom(-,H_T)$ and
$\Hom(G_T,-)$ by varying $T$.

It turns out that the group scheme $\innHom(G,H)$ need not be flat
(or finite) even if both $G$ and $H$ are flat (or finite). We show
this by giving one example in each case.\\

We can generalize the definition of inner $\Hom$ as follows. Define
$\innMult(G_1\times\dots\times G_n,H)$ to be the group scheme
representing the functor
$$\Scheme_S\longto\Ab,\qquad T\mapsto \Mult(G_{1,T}\times\dots\times G_{n,T},H_T).$$

The conditions under which this group exists are identical to those
for\\ $\innHom(G,H)$, i.e., flatness and finiteness of $G_i$. It is
affine or of finite type if $H$ has these properties.\\

Then we study the group of multilinear morphisms
$\Mult(G_1\times\dots\times G_n,H)$. Consider the case where
$G:=G_1=\dots=G_n$ and write $G^n$ for the product of $n$ copies of
$G$. The group $\Mult(G^n,H)$ has two distinguished subgroups,
namely, the group $\Sym(G^n,H)$ of symmetric multilinear morphisms
and the group $\Alt(G^n,H)$ of alternating multilinear morphisms.
The first one is the group of multilinear morphisms that are
invariant under the obvious action of the symmetric group $S_n$ on
$G^n$ and the second one the group of multilinear morphisms that
vanish when two factors are equal.

In the same way that we construct $\innMult(G_1\times\dots\times
G_n,H)$ from $\Mult(G_1\times\dots\times G_n,H)$, we can
``schematize'' the groups $\Sym(G^n,H)$ and $\Alt(G^n,H)$ and obtain
$\innSym(G^n,H)$ and $\innAlt(G^n,H)$, in order to take into account
the behavior of these groups over different base schemes.\\

In section 1, we establish the following propositions, which show
that our definitions lead to a coherent theory.\\

\textbf{Proposition \ref{prop2}.} \emph{Let
$G_1,\,\dotsc,\,G_r,\,H_1,\,\dotsc,\,H_s,\,F$ be commutative group\\
schemes over a base scheme $S$. We have a natural isomorphism
$$\Mult(G_1\times\dotsb\times G_r,\innMult(H_1\times\dotsb\times
H_s,F))\cong$$ $$\Mult(G_1\times\dotsb\times G_r\times
H_1\times\dotsb\times H_s,F)$$ functorial in all arguments.}\\

In particular we have $\Mult(F\times G,H)\cong\Hom(F,\innHom(G,H))$.
We could therefore take this isomorphism for the definition of
$\innHom(G,H)$, i.e., $\innHom(G,H)$ is the unique group scheme such
that we have a natural isomorphism $$\Mult(-\times
G,H)\cong\Hom(-,\innHom(G,H)).$$ It also shows how naturally
multilinear morphisms arise when one is looking at the homomorphisms
between group schemes. The next
important result is a generalization of Proposition \ref{prop2}:\\

\textbf{Proposition \ref{prop3}.} \emph{Let
$G_1,\,\dotsc,\,G_r,\,H_1,\,\dotsc,\,H_s,\,F$ be commutative group\\
schemes over a base scheme $S$. We have a natural isomorphism
$$\innMult(G_1\times\dotsb\times G_r,\innMult(H_1\times\dotsb\times
H_s,F))\cong$$ $$\innMult(G_1\times\dotsb\times G_r\times
H_1\times\dotsb\times H_s,F)$$ functorial in all arguments.}\\

Then, we give some concrete examples and we show the following
isomorphisms, where the base scheme is $\Spec k$ with $k$ a field of
characteristic $p$ and $\alpha_{p^n}$ denotes the kernel of the
$n^{\text{th}}$ Frobenius of the additive group $\BG_a$ over $k$,
i.e., $\alpha_{p^n}(R)=\{\,a\in R\,|\,a^{p^n}=0\,\}$ for any
$k$-algebra $R$:

\begin{itemize}
\item $\innMult(\alpha_p^n,\BG_m)\cong \BG_a\quad \forall\, n\geq 2.$
\item $\innMult(\alpha_{p^{n_1}}\times\dots\times\alpha_{p^{n_r}},\BG_a)\cong\BG_a^{n_1\dots n_r}.$
\item $\innHom(\alpha_{p^n},\alpha_{p^m})\cong\BG_a^n\quad \text{if}\ m\geq n.$
\item
$\innHom(\alpha_{p^n},\alpha_{p^m})\cong\alpha_{p^m}^{n-m}\times\BG_a^m\quad\text{if}\
m<n.$
\item
$\innSym(\alpha_{p^n}^r,\BG_a)\cong\BG_a^{\binom{n+r-1}{r-1}}.$
\item $\innAlt(\alpha_{p^n}^r,\BG_a)\cong\BG_a^{\binom{n}{r}}\quad
\text{if}\ p>2\quad  (\text{where}\,\binom{n}{r}=0\ \text{if}\
r>n).$
\end{itemize}

In section 2, we make the ``dual'' constructions of the first
section. A multilinear morphism from $G_1\times\dots\times G_n$ to a
commutative group scheme is not a homomorphism of group schemes and
therefore is not a morphism in the category of group schemes, but we
would like to work inside this category. Thus, we should somehow
look at these multilinear morphisms inside this category, that is,
we should replace $G_1\times\dots\times G_n$ by a commutative group
scheme such that for any commutative group scheme $H$ and any
multilinear morphism from the product $G_1\times\dots\times G_n$ to
$H$, there is a unique homomorphism from this new commutative group
scheme to $H$, that satisfies a certain universal property. This is
possible thanks to the tensor product of $G_1,\dots,G_n$. Let the
tensor product of commutative group schemes $G_1,\dots,G_n$ over $S$
be a commutative group scheme $G_1\otimes\dots\otimes G_n$ together
with a ``universal'' multilinear morphism $\phi:G_1\times\dots\times
G_n\to G_1\otimes\dots\otimes G_n$ that yields an isomorphism
$$\Hom(G_1\otimes\dots\otimes G_n,H)\cong \Mult(G_1\times\dots\times
G_n,H),\qquad \psi\mapsto \psi\circ\phi$$ for all commutative group
schemes $H$ over $S$. This universal property determines the tensor
product up to unique isomorphism, if it exists. Theorem 4.3 in
\cite{Pink1} says that the tensor product exists and is pro-finite
if the base scheme $S$ is the spectrum of a field and the $G_i$ are
finite over $S$, and with the notations of the first section it is
isomorphic to the inverse limit $\invlim G_{\alpha}^*$ where
$G_{\alpha}$ runs through all finite subgroup schemes of
$\innMult(G_1\times\dots\times G_n,\BG_m)$. By abuse of notation, we
can write this inverse limit as $\innMult(G_1\times\dots\times
G_n,\BG_m)^*$. This shows that all information about the tensor
product $G_1\otimes\dots\otimes G_n$ and hence about multilinear
morphisms from $G_1\times\dots\times G_n$ to commutative group
schemes can be read off from the group of multilinear morphisms
$\Mult(G_{1,T}\times\dots\times G_{n,T},\BG_{m,T})$ for all
extensions $T\to S$ of the base scheme.

Despite our expectations, the construction of the tensor product
does not commute with the base change, that is, we don't have in general
$(G\otimes H)_T\cong G_T\otimes H_T$ for an $S$-scheme $T$. This makes the calculations more difficult.\\

In a similar fashion, we define the symmetric power $S^nG$,
respectively the alternating power $\Lambda^nG$, of a commutative
group scheme $G$ over $S$ to be the unique commutative group schemes
that characterize, in the same way as the tensor product, the group
$\Sym(G^n,H)$, respectively $\Alt(G^n,H)$ for all commutative group
schemes $H$ over $S$. Again, if $S$ is the spectrum of a field and
$G$ is finite over $S$, these group schemes exist, are pro-finite
and constructed as quotients of the $n$-fold tensor product
$G^{\otimes n}$, similar to the same constructions
for modules over commutative rings.\\

Then we do some explicit calculations and show the following
isomorphisms for $n>1$, where $\BG_a^*:=\invlim_iG_i^*$ and $G_i$
runs through all finite subgroup schemes of $\BG_a$:
\begin{itemize}
\item $S^n\alpha_p\cong\alpha_p^{\otimes
n}\cong\BG_a^*$
\item $\Lambda^n\alpha_p\cong\alpha_p^{\otimes
n}$ if $p=2.$
\item $\Lambda^n\alpha_p=0$ if $p>2$.
\end{itemize}

And more generally:

\begin{itemize}
\item $\innSym(\alpha_p^n,H)=\innMult(\alpha_p^n,H)$.
\item $\innAlt(\alpha_p^n,H)=\innMult(\alpha_p^n,H)$ if $p=2$.
\item $\innAlt(\alpha_p^n,H)=0$ if $p>2$.
\end{itemize}

For the remainder of section 2 we work on alternating multilinear
morphisms and the alternating powers. Our main results are:\\

\textbf{Theorem \ref{lem7}.} \emph{Assume that $0\to G'\arrover{\iota}
G\arrover{\pi} G''\to 0$ is a short exact sequence of commutative
group schemes. Let $m$ be a non negative integer and write
$m=m'+m''$ with non negative integers $m'$ and $m''$. Consider the
diagram
$$\xymatrix{
\Alt(G^m,H)\ar[r]^{\rho\qquad}&\Alt(G'^{m'}\times G^{m''},H)\\
&\Alt(G'^{m'}\times G''^{m''},H)\ar@{^{ (}->}[u]_{\pi^*}}$$ where
$\rho$ is the restriction map.
\begin{itemize}
\item[\emph{(a)}] If $\Lambda^{m''+1}G''=0$, then $\rho$ is injective.
\item[\emph{(b)}] If $\Lambda^{m'+1}G'=0$, then $\rho$ factors through $\pi^*$.
\item[\emph{(c)}] If both conditions hold, then there is a natural epimorphism
$$\zeta:\Lambda^{m'}G'\otimes\Lambda^{m''}G''\onto\Lambda^mG.$$
\item[\emph{(d)}] If furthermore the sequence is split, then the epimorphism
$\zeta$ is an isomorphism.
\end{itemize}}

We know that this theorem is true for modules of finite length over
local rings. The condition $\Lambda^mM=0$ for a such module $M$ is
guaranteed if $m$ is greater than the length of $M$. One would
desire that the same thing holds for commutative group schemes.
Restricting to local-local commutative group schemes over a base
field of odd characteristic $p$, we show:\\

\textbf{Proposition \ref{prop7}.} \emph{Let $G$ be a local-local
commutative group scheme of order $p^n$ with $p$ an odd prime
number. We have:
\begin{itemize}
\item[\emph{(a)}] $\Lambda^m G=0$ for all $m>n$.
\item[\emph{(b)}] $\Lambda^n G$ is a quotient  of $\alpha_p^{\otimes n}$.
\end{itemize}}

We see that for a local-local commutative group scheme of order
$p^n$ for an odd prime number $p$, the exponent $n$ plays somehow
the role of the length of modules of finite length.
Another example that shows this analogy is:\\

\textbf{Corollary \ref{cor1}.} \emph{Let $G$ and $H$ be local-local
commutative group schemes of order $p^n$ and $p^m$ respectively,
with $p$ an odd prime number. Then we have a natural isomorphism
$$\Lambda^{n+m}(G\oplus H)\cong\Lambda^nG\otimes\Lambda^mH.$$}

Finally, we have the following important result:\\

\textbf{Proposition \ref{prop12}.} \emph{Let the base scheme be
$\Spec k$ for a perfect field $k$ of odd characteristic $p$, and $n$
a positive integer. Then there is an isomorphism
$$\Lambda^n\alpha_{p^n}\cong\alpha_p.$$}

The results proved in this paper may hold in a more general context
(see Remark \ref{rem14}), however, we do not intend to state them
with minimal hypotheses.\\

\textbf{Acknowledgements.} The ideas and concepts of the multilinear
theory of commutative group schemes used here are due to Prof. Dr.
Pink who first introduced them, and in this paper we further develop
these ideas. I would like to acknowledge the help of these people in
the course of writing this paper: Thanks are due to Prof. Dr.
Testerman for reading and commenting an earlier version of this
paper, and to Prof. Dr. Bayer and Prof. Dr. Ojanguren for their
availability and
suggestions. I am specially thankful to Prof. Dr. Pink for suggesting the subject of this work and for his helpfulness and advice.\\

\textbf{Conventions.} We suppose some familiarity with the
elementary theory of schemes and group schemes. Throughout the
paper, all schemes are assumed to be separated and quasi-compact. We
usually consider schemes over a fixed base scheme $S$, and in this
case morphisms and fiber products are taken over $S$, unless
otherwise noted. The pullback of a scheme $X$ over $S$ via any
morphism $T\to S$ is denoted by $X_T$. When there is no ambiguity,
we write $\BG_a, \BG_m$ or $\alpha_{p^n}$
instead of $\BG_{a,S}, \BG_{m,S}$ or $\alpha_{p^n,S}$.\\

\section{Inner homs and multilinear morphisms}

\begin{dfn}
\label{dfn1}
Let $G$ and $H$ be commutative group schemes over a
base scheme $S$. Define a contravariant functor $\innHom(G,H)$ from
the category of $S$-schemes to the category of abelian groups as
follows:
$$T\mapsto \innHom(G,H)(T):=\Hom_T(G_T,H_T).$$
If this functor is representable by a group scheme over $S$, that
group scheme is also denoted by $\innHom(G,H)$ and is called the
\emph{inner $\Hom$ from $G$ to $H$.}\xqed{\blacktriangle}
\end{dfn}

\begin{rem}
\label{rem1}
According to Theorem 3.10 in \cite{Pink1}, if $G$
is finite and flat over $S$, then $\innHom(G,H)$ is representable
and if in addition $H$ is affine, resp. of finite type over $S$,
then $\innHom(G,H)$ has the same property. So, in order to assure
the existence of $\innHom(G,H)$, in the sequel, every time we write
$\innHom(G,H)$, we assume that $G$ is finite and flat over the base
scheme, without explicitly mentioning it.\xqed{\lozenge}
\end{rem}

\begin{prop}
\label{prop1}
 Let $H$ be an affine commutative group scheme over a
field $k$. Then the functors $\innHom(-,H)$ and $\innHom(H,-)$ from
the category of affine commutative group schemes over $k$ to itself
are left exact.
\end{prop}

\begin{proof}[\textsc{Proof}. ]
Suppose that $0\to N\arrover{i}G\arrover{\pi}Q\to 0$ is a short
exact sequence of affine group schemes over a field $k$ and denote
by $A, B$ the Hopf algebras representing $G, Q$ and by $I_B$ the
augmentation ideal of $B$. Then the Hopf algebra representing $N$ is
$A/(I_B\cdot A)$. Let $R$ be a $k$-algebra. Since it is flat over
$k$, we have an injection $B\otimes_kR\into A\otimes_kR$ and
therefore $G_R\arrover{\pi_R}Q_R$ is a quotient morphism. We have
also that $(I_B\cdot A)\otimes_kR=(I_B\otimes_kR)\cdot
(A\otimes_kR)$ and so by flatness we have $(A/(I_B\cdot
A))\otimes_kR\cong A\otimes_kR/((I_B\cdot
A)\otimes_kR)=A\otimes_kR/(I_B\otimes_kR)(A\otimes_kR)$. It implies
that $N_R$ is the kernel of $G_R\arrover{\pi_R}Q_R$. Consequently
the short sequence $0\to N_R\arrover{i_R}G_R\arrover{\pi_R}Q_R\to 0$
is exact. Now, fix an affine commutative group scheme $H$. We show
that the sequence
$$0\to\innHom(Q,H)\arrover{\pi^*}\innHom(G,H)\arrover{i^*}\innHom(N,H)$$
is exact. It is equivalent to the exactness of the sequence
$$0\to\innHom(Q,H)(R)\arrover{\pi^*}\innHom(G,H)(R)\arrover{i^*}\innHom(N,H)(R)$$
for every $k$-algebra $R$, i.e., the exactness of the sequence
$$0\to\Hom_R(Q_R,H_R)\arrover{\pi^*}\Hom_R(G_R,H_R)\arrover{i^*}\Hom_R(N_R,H_R).$$
Assume we have shown that for any homomorphism $\phi:G_R\to H_R$
such that $\phi\circ i_R=0$, then there exists a unique homomorphism
$\psi:Q_R\to H_R$ with $\phi=\psi\circ\pi_R$, i.e., the following
diagram is commutative
$$\xymatrix{
0\ar[r]&N_R\ar[r]^{i_R}\ar[dr]^0&G_R\ar[r]^{\pi_R}\ar[d]^{\phi}&Q_R\ar[r]\ar[dl]^{\exists
!\ \psi}&0\\
&&H_R.&&}$$ Then the exactness is clear; indeed, pick a morphism
$f:Q_R\to H_R$ with $f\circ\pi_R=0$ then putting $\phi:=0$ the zero
morphism, there are two morphisms $Q_R\to H_R$, namely $f$ and the
zero morphism, whose composition with $\pi_R$ are $\phi$ and from
the assumption they should be equal. This shows the injectivity of
$$\Hom_R(Q_R,H_R)\arrover{\pi^*_R}\Hom_R(G_R,H_R).$$ Clearly we have
$\image \pi^*_R\subset \kernel i^*_R$. Let $g:G_R\to H_R$ be an
element of $\kernel i^*_R$, i.e., $g\circ i^*_R=0$, then according
to the assumption there is a $\psi:Q_R\to H_R$ with
$\pi_R\circ\psi=g$, or in other words $g=\pi^*_R(\psi)$ and thus
$\kernel i^*_R\subset\image \pi^*_R$.\\

It is thus sufficient to show that the assumption holds. But this is
obvious, since as we proved above, the morphism
$G_R\arrover{\pi_R}Q_R$ is the cokernel of the injection $N_R\into
G_R$ in the category of affine commutative group schemes.

Similarly, the fact that $N_R\arrover{i_R}G_R$ is the kernel of the
quotient morphism $G_R\arrover{\pi_R}Q_r$ implies that given any
homomorphism $\phi:H_R\to G_R$ with trivial composition
$\pi_R\circ\phi$ there is a unique homomorphism $\psi:H_R\to N_R$
such that the following diagram is commutative
$$\xymatrix{
0\ar[r]&N_R\ar[r]^{i_R}&G_R\ar[r]^{\pi_R}&Q_R\ar[r]&0\\
&&H_R\ar[ul]^{\exists!\ \psi}\ar[u]^{\phi}\ar[ur]^0.&&}$$ And this
implies as above the exactness of the following short sequence
$$0\to
\Hom_R(H_R,N_R)\arrover{i^*_R}\Hom_R(H_R,G_R)\arrover{\pi^*_r}\Hom_R(H_R,Q_R)$$for
every $k$-algebra $R$, and consequently the following sequence of
group schemes is exact
$$0\to\innHom(H,N)\arrover{i^*}\innHom(H,G)\arrover{\pi^*}\innHom(H,Q).$$

\end{proof}

A natural question that one may ask is to know to what extent $\innHom(G,H)$ shares the properties of $G$ and $H$.
Examples of such properties are finiteness or flatness. It is quite
easy to see that $\innHom(\alpha_p,\alpha_p)\cong\BG_a$ (and we will
give a detailed proof later), so we observe that despite the
finiteness of $\alpha_p$, the group scheme
$\innHom(\alpha_p,\alpha_p)$ is not finite and thus, this property
is not preserved by the construction of inner $\Hom$. In the
following example, we show that in fact, the flatness also has this
''defect'' and is not preserved by this construction.

\begin{ex}
\label{ex1}Here we give an example of finite flat commutative
group schemes over a $\Lambda$-algebra $R$ such that the group
scheme $\innHom (G,H)$ is not flat. We refer the reader to the paper
\cite{OT} for a discussion of group schemes of prime order, their
classification and the definition of $\Lambda$. We know that the
field $\BF_p$ is canonically a $\Lambda$-algebra and therefore any
$\BF_p$-algebra is canonically a $\Lambda$-algebra. Put
$R:=\BF_p[x]$, the polynomial ring in one variable over the field
$\BF_p$. Then any elements $a,\,b\in R$ satisfying $ab=0$ define a
group scheme $G_{a,b}:=\Spec R[y]/(y^p-ay)$ together with the
comultiplication
$$y\mapsto y\otimes1+1\otimes y+b(1-p)^{-1}\cdot
\sum_{i=1}^{p-1}\frac{y^i}{w_i}\otimes\frac{y^{p-i}}{w_{p-i}}$$ and
according to Proposition 3.11 in \cite{Pink1}, if $c,\,d\in R$ are
such that $cd=0$, we have $$ \innHom (G_{a,b},G_{c,d})\cong \Spec
R[y]/(ay^p-cy,dy^p-by)$$ with the comultiplication $$y\mapsto
y\otimes1+1\otimes
y+ad(1-p)^{-1}\cdot\sum_{i=1}^{p-1}\frac{y^i}{w_i}\otimes\frac{y^{p-i}}{w_{p-i}}.$$
The group scheme $G_{0, x}$, represented by the Hopf algebra
$R[y]/(y^p)$, is flat over $R$, because this Hopf algebra is a
torsion-free module over the principal ideal domain $R$ and so is
flat over $R$. But the group scheme $\innHom (G_{0,x},G_{0,x})$ is
represented by the Hopf algebra $R[y]/(xy^p-xy)$ which has the
torsion element $y^p-y$ (which is annihilated by $x$) and therefore
is not flat over $R$. It follows then that $\innHom
(G_{0,x},G_{0,x})$ is not flat over $R$.\xqed{\blacksquare}
\end{ex}

Recall that if $G_1,\dotsc,G_r,H$ are commutative group schemes over
a base $S$, then $\Mult(G_1\times\dotsb\times G_r,H)$ is the group
of all multilinear morphisms from $G_1\times\dotsb\times G_r$ to
$H$, i.e. morphisms that are linear in each factor $G_i$ or
equivalently morphisms which have the property that for any
$S$-scheme $T$ the induced morphism $G_1(T)\times\dotsb\times
G_r(T)\to H(T)$ is multilinear. We can then generalize the
definition of inner $\Hom$ as follows:

\begin{dfn}
\label{dfn2}
Let $G_1, G_2,\dotsc, G_r, H$ be commutative group
schemes over a base scheme $S$. Define a contravariant functor from
the category of $S$-schemes to the category of abelian groups as
follows:
$$T \mapsto\innMult(G_1\times G_2\times\dotsb\times G_r,
H)(T):=\Mult_T(G_{1,T}\times G_{2,T}\times\dotsb\times G_{r,T},
H_T).$$ If this functor is representable by a group scheme over $S$,
we will also denote that group scheme by $\innMult(G_1\times
G_2\times\dotsb\times G_r, H)$.\xqed{\blacktriangle}
\end{dfn}

For any positive integer $r$ we denote by $G^r$ the product of $r$
copies of $G$, and for any $1\leq i,j\leq r$ we let
$\Delta^r_{ij}\subset G^r$ or $\Delta^r_{ij}G$ (if we want to make
explicit the group scheme $G$) denote the closed subscheme defined
by equating the $i^{\text{th}}$ and $j^{\text{th}}$ components.

\begin{dfn}
\label{dfn3}
Let $G$ and $H$ be as above.
\begin{itemize}
\item[\emph{(i)}] A multilinear morphism $G^r\to H$ is called \emph{symmetric}
if it is invariant under permutation of the factors. The group of
all such symmetric multilinear morphisms is denoted $\Sym(G^r,H)$.

\item[\emph{(ii)}] A multilinear morphism $G^r\to H$ is called
\emph{alternating} if its restriction to $\Delta^r_{ij}$ is trivial
for all $1\leq i,j\leq r$. The group of all such alternating
multilinear morphisms is denoted $\Alt(G^r,H)$.\xqed{\blacktriangle}
\end{itemize}
\end{dfn}

\begin{rem}
\label{rem2}
\begin{itemize}
\item[1)] Let $\phi:G^r\to H$ be a multilinear morphism. Then one can see easily that $\phi$ is symmetric if and only if
the induced morphism $\phi(T):G^r(T)=G(T)^r\to H(T)$ is symmetric
for all $S$-schemes $T$.
\item[2)] We have a natural action of the symmetric group $S_r$ on
$G^r$. This action induces an action on the group
$\Mult(G_1\times\dotsb\times G_r,H)$ and the subgroup $\Sym(G^r,H)$
is precisely the subgroup of fixed points, i.e.
$\Sym(G^r,H)=\Mult(G^r,H)^{S_r}$.
\item[3)] Similarly to 1), if $\psi:G^r\to H$ is a multilinear morphism, then
$\psi$ is alternating if and only if $\psi(T):G^R(T)=G(T)^r\to H(T)$
is alternating for all $S$-schemes $T$.
\item[4)] The usual calculation shows that any alternating morphism is
antisymmetric, i.e. a permutation of the factors multiplies the
morphism by the sign of the permutation.\xqed{\lozenge}
\end{itemize}
\end{rem}

We can make definitions similar to Definition \ref{dfn2} for the
group of symmetric and alternating multilinear morphisms:

\begin{dfn}
\label{dfn4}
Let $G, H$ be commutative group schemes over $S$. Then
denote by $\innSym(G^r,H)$ and $\innAlt(G^r,H)$ respectively the
contravariant functors
$$T \mapsto \innSym(G^r,H)(T):=\Sym_T(G_T^r,H_T)$$ and
$$T\mapsto \innAlt(G^r,H)(T):=\Alt_T(G_T^r,H_T)$$ respectively. If
$\innSym(G^r,H)$ resp. $\innAlt(G^r,H)$, is representable by a
commutative group scheme, we will also denote this group scheme by
$\innSym(G^r,H)$ resp. $\innAlt(G^r,H)$.
\xqed{\blacktriangle}\end{dfn}

We are now going to prove a general proposition on multilinear
morphisms which will be used throughout the paper, but we first
establish two lemmas:

\begin{lem}
\label{lem1}
If $\innHom(G,H)$ is representable, there is a natural
isomorphism
$$\Mult(G_1\times\dotsb\times
G_r,\innHom(G,H))\cong\Mult(G_1\times\dotsb\times G_r\times G,H),$$
functorial in all arguments.
\end{lem}

\begin{proof}[\textsc{Proof}. ]
By the definition of $\innHom(G,H)$, giving a morphism of schemes
$\phi:G_1\times\dotsb\times G_r\to \innHom(G,H)$  is equivalent to
giving a morphism of schemes $\widetilde{\phi}:G_1\times\dotsb\times
G_r\times G\to H$ which is linear in $G$. Since the group structure
of $\innHom(G,H)$ is induced by that of $H$, one sees easily that
$\phi$ is linear in $G_i$ if and only if $\widetilde{\phi}$ is
linear in $G_i$. This completes the proof.
\end{proof}

Now, we give an ''underline'' version of this lemma in order to show
our general result of this type:

\begin{lem}
\label{lem2}
If $\innMult(G_1\times\dotsb\times G_{r-1},
\innHom(G,H))$ and $\innHom(G,H)$ are representable, then
$\innMult(G_1\times\dotsb\times G_r, H)$ is representable and there
is a natural isomorphism
$$\innMult(G_1\times\dotsb\times G_r,
\innHom(G,H))\cong\innMult(G_1\times\dotsb\times G_r\times G,
H)$$functorial in all arguments.
\end{lem}

\begin{proof}[\textsc{Proof}. ]
If we establish the isomorphism, the representability will follow
directly from it, because if two functors are naturally isomorphic
and one is representable, the other is representable too. We show
thus only the isomorphism. We show at first that for any commutative
group schemes $G$ and $H$ over $S$ and any $S$-scheme $T$, we have
$\innHom(G_T,H_T)\cong\innHom(G,H)_T$. Indeed, if $X$ is any
$T$-scheme, then $\innHom(G_T,
H_T)(X)=\Hom_X((G_T)_X,(H_T)_X)=\Hom_X(G_X,H_X)=\innHom(G,H)(X)=\innHom(G,H)_T(X)$.
Now, we have $$\innMult(G_1\times G_2\times\dotsb\times
G_r,H)(T)=\Mult(G_{1,T}\times G_{2,T}\times\dotsb\times
G_{r,T},H_T)$$ and by Lemma \ref{lem1} this is isomorphic to
$$\Mult(G_{1,T}\times G_{2,T}\times\dotsb\times
G_{r-1,T},\innHom(G_{r,T},H_T)).$$ By the above discussion, it is
isomorphic to
$$\Mult(G_{1,T}\times G_{2,T}\times\dotsb\times G_{r-1,T},\innHom(G_r,H)_T)=$$ $$\innMult(G_1\times
G_2\times\dotsb\times G_{r-1},\innHom(G_r,H))(T).$$ This achieves
the proof.
\end{proof}

\begin{rem}
\label{rem3}\begin{itemize}
\item[1)]Assume that $G_1,\dotsc,G_r$ are finite and flat over
$S$. We can show by induction on $r$ that
$\innMult(G_1\times\dotsb\times G_r,H)$ is representable by a
commutative group scheme. If furthermore $H$ is affine or resp. of
finite type, then $\innMult(G_1\times\dotsb\times G_r,H)$ has the
same property. Indeed, if $r=1$ then this is exactly Theorem 3.10 in
\cite{Pink1}. So let $r>1$ and suppose that the statement is true
for $r-1$. By the induction hypothesis,
$\innMult(G_1\times\dotsb\times G_{r-1},\innHom(G_r,H))$ is
representable and has the same properties (affineness or being of
finite type) of $\innHom(G,H)$ which has itself the same properties
as $H$ according to Theorem 3.10 in \cite{Pink1}. From Lemma
\ref{lem2}, it follows that
$$\innMult(G_1\times\dotsb\times G_{r-1},\innHom(G,H))\cong
\innMult(G_1\times\dotsb\times G_r,H).$$ Hence, the right hand side
is representable and has the same properties as $H$.
\item[2)] Let $G$ be finite and flat over $S$. By definition \ref{dfn4}, it is clear that the functors
$\innSym(G^r,H)$ and $\innAlt(G^r,H)$ are
subfunctors of the representable functor $\innMult(G^r,H)$. Since
the conditions defining these subfunctors are closed conditions
(given by equations), they are represented by closed subgroup
schemes.
\item[3)] We will thus make the assumption that every time we use $\innMult(G_1\times\dotsb\times G_1,H), \innSym(G^r,H)$ or
$\innAlt(G^r,H)$, the group schemes $G_1,\dotsc,G_r$ and $G$ are
finite and flat over $S$ and we will no longer worry about the
representability of these functors.\xqed{\lozenge}
\end{itemize}
\end{rem}

Here is the desired proposition:
\begin{prop}
\label{prop2}
Let $G_1,\,\dotsc,\,G_r,\,H_1,\,\dotsc,\,H_s,\,F$ be
commutative group\\ schemes over a base scheme $S$. We have a
natural isomorphism
$$\Mult(G_1\times\dotsb\times G_r,\innMult(H_1\times\dotsb\times
H_s,F))\cong\Mult(G_1\times\dotsb\times G_r\times
H_1\times\dotsb\times H_s,F)$$ functorial in all arguments.
\end{prop}

\begin{proof}[\textsc{Proof}. ]
We prove this proposition by induction on $s$. If $s=1$, then it is
exactly the Lemma \ref{lem1}. So assume that $s>1$ and that the
proposition is true for $s-1$. We have a series of isomorphisms:
$$\Mult(G_1\times\dotsb\times G_r\times H_1\times\dotsb\times
H_s,F)\overset{\ref{lem1}}{\cong}$$
$$\Mult(G_1\times\dotsb\times G_r\times H_1\times\dotsb\times H_{s-1},\innHom(H_s,F))\overset{\text{ind.}}{\underset{\text{hyp.}}{\cong}}$$
$$\Mult(G_1\times\dotsb\times G_r,\innMult(H_1\times\dotsb\times H_{s-1},\innHom(H_s,F)))\overset{\ref{lem2}}{\cong}$$
$$\Mult(G_1\times\dotsb\times G_r,\innMult(H_1\times\dotsb\times H_s,F)).$$
\end{proof}

\begin{rem}
\label{rem4} Let $S$ be a scheme and $G_1\dotsc,G_r,H_1\dotsc,H_s,G,H$ and $F$
commutative group schemes over $S$. There is a natural
action of the symmetric group $S_n$ on $H^n$ that induces an action
on the group scheme $\innMult(H^n,F)$ which itself induces an action
on the group
$$\Mult(G_1\times\dotsb\times G_r,\innMult(H^n,F)).$$ We also have a
natural action of this group on the group
$$\Mult(G_1\times\dotsb\times G_1\times H^n,H).$$ One checks
that the isomorphism in the proposition is invariant under the
action of $S_n$. Similarly, we have an action of the symmetric group
$S_m$ on
$$\Mult(G^m,\innMult(H_1\times\dotsb\times H_s,F))\,\ \text{and}\,\
\Mult(G^m\times H_1\times\dotsb\times H_s,F)$$ induced by its action
on $G^m$. Again, one can easily verify that the isomorphism in the
proposition is invariant under this action of $S_m$.\xqed{\lozenge}\end{rem}

In the same way that Lemma \ref{lem2} follows from Lemma \ref{lem1}
the following proposition can be deduced from Proposition
\ref{prop2}; we will thus omit the proof:

\begin{prop}
\label{prop3} Let $G_1,\,\dotsc,\,G_r,\,H_1,\,\dotsc,\,H_s,\,F$ be
commutative group\\ schemes over a base scheme $S$. We have a
natural isomorphism
$$\innMult(G_1\times\dotsb\times G_r,\innMult(H_1\times\dotsb\times
H_s,F))\cong\innMult(G_1\times\dotsb\times G_r\times
H_1\times\dotsb\times H_s,F)$$ functorial in all arguments.\qed
\end{prop}

Fix a base scheme $S=\Spec k$ for a field $k$ and let
$G_1,\,G_2,\,\dotsc,\, G_r$ be finite commutative group schemes and
let $H$ be a commutative group scheme. Then by Proposition
\ref{prop2}, we have an isomorphism that is functorial in $H$:
$$\Mult (H\times G_1\times G_2\times\dotsb\times G_r,\BG_m)\cong \Hom(H,\innMult(G_1\times\dotsb\times G_r,\BG_m)).$$

Let us write $\widetilde{G}$ for $\innMult(G_1\times\dotsb\times
G_r,\BG_m)$. Then this means that we have a natural isomorphism
$$\Mult(-\times G_1\times\dotsb\times
G_r,\BG_m)\arrover{\tau}\Hom(-,\widetilde{G})=\widetilde{G}(-)$$ or
in other words, the group scheme $\widetilde{G}$ represents the
functor $$\Mult(-\times G_1\times\dotsb\times G_r,\BG_m).$$ Assume
that we have a multilinear morphism $$\phi:H\times G_1\times\dotsb
G_r\to \BG_m,$$ then by functoriality, we have a commutative
diagram:

$$\xymatrix{\Mult(\widetilde{G}\times G_1\times\dotsb\
G_r,\BG_m)\ar[d]_{(\tau_H(\phi)\times 1\times\dotsb\times
1)^*}\ar[rr]^{\text{\qquad\qquad}\tau_{\widetilde{G}}}_{\text{\qquad\qquad}\cong}
& &
\Hom(\widetilde{G},\widetilde{G})\ar[d]_{\tau_H(\phi)}\\
\Mult(H\times G_1\times\dotsb\times
G_r,\BG_m)\ar[rr]^{\text{\qquad\qquad}\tau_H}_{\text{\qquad\qquad}\cong}
& & \Hom(H,\widetilde{G})}$$

where $\tau_H(\phi)^*(f)=f\circ\tau_H(\phi)$ and
$$(\tau_H(\phi)\times 1\times\dotsb\times 1)^*(g)=(\tau_H(\phi)\times
1\times\dotsb\times 1)\circ g.$$

\begin{prop}
\label{prop6} Assume that
$\psi_{\widetilde{G}}\in\Mult(\widetilde{G}\times
G_1\times\dotsb\times G_r,\BG_m)$ is such that
$\tau_{\widetilde{G}}(\psi_{\widetilde{G}})=\Id_{\widetilde{G}}$.
Then the pair $(\psi_{\widetilde{G}},\widetilde{G})$ satisfies the
following universal property: Given any multilinear morphism
$\phi:H\times G_1\times\dotsb\times G_r\to\BG_m$ there is a unique
homomorphism $\widetilde{\phi}:H\to\widetilde{G}$ such that the
following diagram commutes
$$\xymatrix{
H\times G_1\times\dotsb\times G_r\ar[dr]_{\widetilde{\phi}\times 1\times\dotsb\times 1}\ar[rr]^{\phi}&&\BG_m\\
&\widetilde{G}\times G_1\times\dotsb\times
G_r\ar[ur]_{\psi_{\widetilde{G}}}.&}$$\\
\end{prop}

\begin{proof}[\textsc{Proof}. ]
We show that the unique morphism $\widetilde{\phi}$ is
$\tau_H(\phi)$. Since
$\tau_{\widetilde{G}}(\psi_{\widetilde{G}})=\Id_{\widetilde{G}}$, we
have by the commutative diagram before the proposition
$$\tau_H((\tau_H(\phi)\times 1\times\dotsb\times
1)\circ\psi_{\widetilde{G}})=\tau_H((\tau_H(\phi)\times
1\times\dotsb\times 1)^*(\psi_{\widetilde{G}}))=$$
$$\tau_H(\phi)^*(\tau_{\widetilde{G}}(\psi_{\widetilde{G}}))=\tau_H(\phi)^*(\Id_{\widetilde{G}})=\tau_H(\phi).$$
But, $\tau_H$ is injective and therefore $\phi=\tau_H(\phi)\times
1\times\dotsb\times 1\circ\psi_{\widetilde{G}}.$ This shows the
existence of $\widetilde{\phi}$.

Now, if we have a morphism $f:H\to\widetilde{G}$ with the property
that $\phi=f\times 1\times\dotsb\times 1\circ\psi_{\widetilde{G}}$,
then by surjectivity of $\tau_H$, there exists a multilinear
morphism $\phi':H\times G_1\times\dotsb\times G_r\to\BG_m$ such that
$\tau_H(\phi')=f$, and from what we have shown above
$$\phi'=\tau_H(\phi')\times 1\times\dotsb\times 1\circ\psi_{\widetilde{G}}=f\times 1
\times \dotsb\times 1\circ\psi_{\widetilde{G}}=\phi$$ Thus
$\tau_H(\phi)=\tau_H(\phi')=f$. This proves the uniqueness of
$\widetilde{\phi}$.
\end{proof}

\begin{dfn}
\label{dfn7} We call the group scheme
$\innMult(G_1\times\dotsb\times G_r,\BG_m)$, resp. the multilinear
morphism $\psi:\innMult(G_1\times\dotsb\times G_r,\BG_m)\times
G_1\times\dotsb\times G_r\to \BG_m$ defined in Proposition
\ref{prop6}, \emph{the universal group scheme resp. the universal
multilinear morphism associated to $G_1,\dotsc,G_r$.}
\xqed{\blacktriangle}\end{dfn}

\begin{lem}
\label{lem3} Let $G,H$ be commutative group schemes over a base
scheme $S$ and let $\Gamma$ be a finite group acting on $G$. Then we
have a natural isomorphism
$$\Hom(H,G)^{\Gamma}\cong\Hom(H,G^{\Gamma}),$$ where $G^{\Gamma}$ is
the subgroup scheme of fixed points, in other words,
$G^{\Gamma}(T)=G(T)^{\Gamma}$ for any $S$-scheme $T$, where
$G(T)^{\Gamma}$ is the subgroup of fixed points of the abelian
$\Gamma$-group $G(T)$ and the action of $\Gamma$ on $\Hom(H,G)$ is
induced by its action on $G$. More precisely, the image of the
inclusion $\Hom(H,G^{\Gamma})\into \Hom(H,G)$ is the group of fixed
points $\Hom(H,G)^{\Gamma}$.
\end{lem}

\begin{proof}[\textsc{Proof}. ]
Let $\phi:H\to G^{\Gamma}$ and $\gamma\in\Gamma$ be given. The image
of $\phi$ under the inclusion in the lemma is the composition
$H\arrover{\phi}G^{\Gamma}\into G$ and under the action of $\gamma$
on $\Hom(H,G)$ it maps to the morphism $H\arrover{\phi}
G^{\Gamma}\into G\arrover{\gamma\cdot} G$. But by definition of
$G^{\Gamma}$, we have that the composition $G^{\Gamma}\into
G\arrover{\gamma\cdot} G$ is the same as the composition
$G^{\Gamma}\into G$ and therefore $\gamma$ fixes the image of $\phi$
and hence it is an element of $\Hom(H,G)^{\Gamma}$. We have thus an
inclusion $\Hom(H,G^{\Gamma})\subset \Hom(H,G)^{\Gamma}$, where we
have identified $\Hom(H,G^{\Gamma})$ with its image.

Now, assume that we have a morphism $\psi:H\to G$ which lies inside
the group of fixed points. This means that the composition
$\gamma\circ\psi$ for any $\gamma\,\cdot:G\to G$ is equal to $\psi$
and therefore $\psi$ must factor through $G^{\Gamma}$. This gives
the inclusion $\Hom(H,G)^{\Gamma}\subset \Hom(H,G^{\Gamma})$ and the
lemma is proved.
\end{proof}

We are now going to apply this lemma to the particular case, where
the acting group is the symmetric group $S_n$ which acts on the
group scheme $\innMult(G^n,F)$ with $G$ and $F$ two commutative
group schemes. The lemma states that we have an isomorphism

$$\Hom(H,\innMult(G^n,F)^{S_n})\cong \Hom(H,\innMult(G^n,F))^{S_n}$$

By Definitions \ref{dfn3} and \ref{dfn4} and Remark \ref{rem2},
$\innSym(G^n,F)$ is exactly the group of fixed points
$\innMult(G^n,F)^{S_n}$, and therefore we can rewrite the last
isomorphism as
$$\Hom(H,\innSym(G^n,F))\cong \Hom(H,\innMult(G^n,F))^{S_n}\qquad (\bigstar)$$

We now apply Proposition \ref{prop2} and Remark \ref{rem4}: taking
the fixed points of both sides of the isomorphism in Proposition
\ref{prop2}, we will again get an isomorphism. We can thus apply it
to our situation, and obtain the isomorphism:
$$\Hom(H,\innMult(G^n,F))^{S_n}\cong\Mult(H\times G^n,F)^{S_n}.$$
Combining this with $(\bigstar)$, we have the following proposition:

\begin{prop}
\label{prop4}
With the above notations, there is a natural
isomorphism
$$\Hom(H,\innSym(G^n,F))\cong\Mult(H\times G^n,F)^{S_n}$$
functorial in all arguments.\qed
\end{prop}

\begin{rem}
\label{rem5}
\begin{itemize}
\item[1)]We recall that the action of $S_n$ on the right hand side
consists of permuting the factors of $G^n$ and consequently, the
group $\Mult(H\times G^n,F)^{S_n}$ contains the multilinear
morphisms from $H\times G^n$ to $F$ that are symmetric in $G^n$.
\item[2)]Note that the functoriality of this isomorphism in $H$ implies
that the group scheme $\innSym(G^n,F)$ represents the functor
$\Mult(-\times G^n,F)^{S_n}$.
\item[3)]It is clear that if we change $H\times G^n$ to $G^n\times
H$ the proposition remains valid; we have thus another natural and
functorial isomorphism
$$\Hom(H,\innSym(G^n,F))\cong\Mult(G^n\times H,F)^{S_n}.$$
\end{itemize}
\xqed{\lozenge}\end{rem}

Similar arguments prove the following proposition:
\begin{prop}
\label{prop10} Let $G, H_1,\dots,H_r$ and $F$ be commutative group
schemes. We have a natural isomorphism
$$\Mult(H_1\times\dots\times H_r\times
G^n,F)^{S_n}\cong\Mult(H_1\times\dots\times
H_r,\innSym(G^n,F)).$$\qed
\end{prop}

We can show, with slight modifications of arguments, similar results
concerning the group of alternating multilinear morphisms and in
particular the following proposition:
\begin{prop}
\label{prop11} With above notations we have a natural isomorphism
$$\Alt(H_1\times\dots\times H_r\times
G^n,F)\cong\Mult(H_1\times\dots\times H_r,\innAlt(G^n,F))$$ where
the first group is the group of multilinear morphisms that are
alternating in $G^n$.\qed
\end{prop}

\begin{rem}
\label{rem12} We could show Propositions \ref{prop10} and
\ref{prop11} using the isomorphism given in Proposition \ref{prop2}.
Indeed, under that isomorphism the image of an element of
$\Mult(H_1\times\dots\times H_r\times G^n,F)^{S_n}$ lies inside the
subgroup\\ $\Mult(H_1\times\dots\times H_r,\innSym(G^n,F))$ of
$\Mult(H_1\times\dots\times H_r,\innMult(G^n,F))$ and vice versa.
This is what we explained in Remark \ref{rem4}. The same argument
works in the alternating case.\xqed{\lozenge}\end{rem}

Let $G,H$ and $F$ be as above and assume that we are in a situation
where any multilinear morphism from $G^n$ to $F$ is symmetric. Then
we have in particular
$$\Mult(G^n,\innHom(H,F))=\Mult(G^n,\innHom(H,F))^{S_n} \qquad (\blacktriangle)$$
According to Proposition \ref{prop2} we have an isomorphism
$$\Mult(G^n\times H,F)\cong \Mult(G^n,\innHom(H,F)),$$
and referring again to Remark \ref{rem4}, we obtain the following
commutative diagram:
$$\xymatrix{
\Mult(G^n\times H,F)^{S_n}\ar[r]^{\cong\text{\,\quad
}}\ar@{^{ (}->}[d]_i&\Mult(G^n,\innHom(H,F))^{S_n}\ar@{^{ (}->}[d]^j\\
\Mult(G^n\times H,F)^{\text{\quad}}\ar[r]^{\cong\text{\,\quad
}}&\Mult(G^n,\innHom(H,F)).}$$ We deduce from $(\blacktriangle)$
that $j$ is an equality and therefore $i$ is an equality as well.
Another use of Proposition \ref{prop2} and Proposition \ref{prop4}
gives the following commutative diagram:

$$\xymatrix{
\Mult(G^n\times
H,F)^{S_n}\ar@{^{ (}->}[d]_i^{\shortparallel}\ar[r]^{\cong}_{\ref{prop4}}&\Hom(H,\innSym(G^n,F))\ar@{^{ (}->}[d]_u\\
\Mult(G^n\times
H,F)^{\text{\quad}}\ar[r]^{\cong}_{\ref{prop2}}&\Hom(H,\innMult(G^n,F)).}$$
As we have seen, $i$ is an equality, which implies that $u$ is also
an equality. Since this is true for any commutative group scheme
$H$, it follows that $\innSym(G^n,H)=\innMult(G^n,F)$.

Suppose that the base scheme $S$ is defined over $\BF_p$ with $p>2$
and that $T$ is any $S$-scheme. Take an alternating multilinear
morphism $\psi:G_T^n\to F_T$, i.e. an element of
$\innAlt(G^n,F)(T)$. Since $\innSym(G^n,H)=\innMult(G^n,F)$ and so
$\innSym(G^n,H)(T)=\innMult(G^n,F)(T)$, this $\psi$ should be
symmetric. As it is also alternating and the characteristic is odd,
we deduce that $\psi$ is the zero morphism. Recapitulating, we have:

\begin{prop}
\label{prop5}
Suppose that $G$ and $F$ are commutative group schemes
over a base scheme $S$. Suppose also that any multilinear morphism
from $G^n$ to any commutative group scheme is symmetric. Then if
$\,\innMult(G^n,F)$ exists, we have
\begin{itemize}
\item $\innSym(G^n,F)=\innMult(G^n,F)$ and
\item $\innAlt(G^n,F)=0$ if $S$ is defined over $\BF_p$ with $p>2$.\qed
\end{itemize}
\end{prop}

\begin{prop}
\label{ex2}
\emph{Let $n\geq 2$, then there is an isomorphism
$$\innMult(\underbrace{\alpha_p\times\dotsb\times\alpha_p}_{n \text{\
times}},\BG_m)\cong\BG_a.$$}
\end{prop}

\begin{proof}[\textsc{Proof}. ]
We first show that $\innHom(\alpha_p,\BG_a)\cong\BG_a$. Indeed, let
$R$ be a $k$-algebra and $\phi:k[X]\otimes_kR=R[X]\to
k[X]/(X^p)\otimes_kR=R[X]/(X^p)$ be an $R$-Hopf algebra
homomorphism, write $x$ for the image of $X$ in $R[X]/(X^p)$, and
let $\phi(X)=a_0+a_1x+\dotsc+a_{p-1}x^{p-1}\in R[X]/(X^p)$. Then
being a Hopf algebra homomorphism amounts to saying that
$$\sum_{i=0}^{p-1}a_i(1\otimes x^i+x^i\otimes1)=\sum_{i=0}^{p-1}a_i(1\otimes x+x\otimes1)^i.$$ Since $\{x^j\otimes
x^l\}_{j,l=0}^{p-1}$ form an $R$-basis of
$R[X]/(X^p)\otimes_kR[X]/(X^p)$, $a_i$ should be zero for $i\neq1$.
We have therefore $\phi(X)=a_1x$ for an element $a_1\in R$.
Consequently, $R$-Hopf algebra homomorphisms from $R[X]$ to
$R[X]/(X^p)$ are of the form $X\mapsto r\cdot x$, and any such
morphism is an $R$-Hopf algebra homomorphism. Moreover, the sum of
two such homomorphisms $X\mapsto r\cdot x$ and $X\mapsto s\cdot x$
is $X\mapsto (r+s)\cdot x$. The parameter $r$ thus defines an
isomorphism $\innHom(\alpha_p,\BG_a)\cong\BG_a$.\\

Secondly, since $(rx)^p=0$ in $R[X]/(X^p)$, we find that any
homomorphism $\alpha_{p,R}\to\BG_{a,R}$ factors through
$\alpha_{p,R}\subset\BG_{a,R}$. Therefore
$$\innHom(\alpha_p,\alpha_p)=\innHom(\alpha_p,\BG_a)\cong\BG_a.$$

We also have canonical isomorphisms, $\alpha_p^*\cong\alpha_p$ and
$\innHom(\alpha_p,\BG_m)\cong\alpha_p^*$.\\
Finally, putting all this together and using Lemma \ref{lem2}, we
obtain
$$\innMult(\alpha_p^{n},\BG_m)\cong\innMult(\alpha_p^{n-1},\innHom(\alpha_p,\BG_m))\cong\innMult(\alpha_p^{n-1},\alpha_p^*)\cong$$
$$\innMult(\alpha_p^{n-1},\alpha_p)\cong\innMult(\alpha_p^{n-1},\innHom(\alpha_p,\alpha_p))\cong\innMult(\alpha_p^{n-1},\BG_a)\cong$$
$$\innMult(\alpha_p^{n-2},\innHom(\alpha_p,\BG_a))\cong\innMult(\alpha_p^{n-2},\BG_a)\cong\dotsb\cong\innHom(\alpha_p,\BG_a)\cong\BG_a.$$
\end{proof}

\begin{prop}
\label{ex3}
\emph{Let $n, \, m\geq 1$  be natural numbers, we have
$\innHom(\alpha_{p^n},\alpha_{p^m}) \cong $}

\begin{itemize}
\item $\BG_a^n=\overbrace{\BG_a\times\dotsb\times\BG_a}^{n \text{\ \emph{ times}}}$  \emph{if} $m\geq n$
\item $\alpha_{p^m}^{n-m}\times\BG_a^m=\underbrace{\alpha_{p^m}\times\dotsb\times\alpha_{p^m}}_{n-m \text{\
\emph{times}}}\times\underbrace{\BG_a\times\dotsb\times\BG_a}_{m
\text{\ \emph{times}}}$ \emph{if} $m<n.$
\end{itemize}
\end{prop}

\begin{proof}[\textsc{Proof}. ]
Let $R$ be a $k$-algebra and $\phi:\alpha_{p^n,R}\to \alpha_{p^m,R}$
be a homomorphism. Then $\phi$ corresponds to a Hopf algebra
homomorphism $R[X]/(X^{p^m})\to R[X]/(X^{p^n})$ which we denote
again by $\phi$. Let $x$ be the class of $X$ in $R[X]/(X^{p^m})$ and
$R[X]/(X^{p^n})$. Write
$\phi(x)=a_0+a_1x+\dotsb+a_{p^n-1}x^{p^n-1}.$ This element of
$R[X]/(X^{p^n})$ should fulfill two conditions, namely
$\phi(x)^{p^m}=0$ and
$$\Delta(\phi(x))=(\phi\otimes\phi)(\Delta(x))=1\otimes\phi(x)+\phi(x)\otimes1.$$
We first exploit the second condition, which gives
$$\Delta(\sum_{i=0}^{p^n-1}a_ix^i)=1\otimes(\sum_{i=0}^{p^n-1}a_ix^i)+(\sum_{i=0}^{p^n-1}a_ix^i)\otimes1.$$ We
thus have
$$\sum_{i=0}^{p^n-1}a_i\Delta(x)^i=\sum_{i=0}^{p^n-1}a_i(1\otimes x+x\otimes1)^i=\sum_{i=0}^{p^n-1}a_i(1\otimes x^i+x^i\otimes1).$$
Since the elements $x^i\otimes x^j$ are linearly independent, we
must have $a_i(1\otimes x+x\otimes 1)^i=a_i(1\otimes
x^i+x^i\otimes1)$ for all $i$, i.e., we have
$$a_i(\sum_{j=0}^{i}\binom{i}{j}x^j\otimes x^{i-j}-1\otimes x^i+x^i\otimes1)=a_i\sum_{j=1}^{i-1}\binom{i}{j}x^j\otimes x^{i-j}=0.$$
If $i$ is not a power of $p$, then by Lucas theorem there is a $j$
with $\binom{i}{j}$ not divisible by $p$ and therefore, from the
linear independence of $x^j\otimes x^{i-j}$ we deduce that $a_i=0$
for these $i$'s, and we can write
$$\phi(x)=a_1x+a_px^p+a_{p^2}x^{p^2}+\dotsb+a_{p^{n-1}}x^{p^{n-1}}.$$ Consider now the first condition. If $m\geq n$,
 this condition is automatically satisfied and any $n$-tuple $(a_1,a_p,\dotsc,a_{p^{n-1}})$ gives rise to a unique homomorphism
 $\alpha_{p^n}\to \alpha_{p^m}$, and one sees easily that the component-wise addition of these $n$-tuples corresponds to the addition of
 homomorphisms $\alpha_{p^n}\to \alpha_{p^m}$. This gives the isomorphism $$\innHom(\alpha_{p^n},\alpha_{p^m})
 \cong \BG_a^n.$$ If $m<n$, the first condition implies that
 $$(a_1x+a_px^p+a_{p^2}x^{p^2}+\dotsb+a_{p^{n-1}}x^{p^{n-1}})^{p^m}=0,$$i.e., $$a_1^{p^m}x^{p^m}+a_p^{p^m}x^{p^{m+1}}+a_{p^2}^{p^m}x^{p^{m+2}}
 +\dotsb+a_{p^{n-1}}^{p^m}x^{p^{m+n-1}}=0.$$ For indices $i$ with $m+i\geq n$ we have $x^{p^{m+i}}=0$, therefore we must have
 $$a_1^{p^m}x^{p^m}+a_p^{p^m}x^{p^{m+1}}+a_{p^2}^{p^m}x^{p^{m+2}}
 +\dotsb+a_{p^{n-m-1}}^{p^m}x^{p^{n-1}}=0$$which implies that $a_{p^i}^{p^m}=0$
for all $0\leq i\leq m-n-1$ and there is no condition on other
$a_i$'s. Consequently, the $n$-tuples $(a_1,a_p,\dotsc,a_{p^{n-1}})$
belong to the group $\alpha_{p^m}^{n-m}(R)\times\BG_a^n(R)$. Again,
the component-wise addition of these tuples corresponds to the
addition in $\Hom(\alpha_{p^n},\alpha_{p^m})$, and therefore we have
an isomorphism
$$\innHom(\alpha_{p^n},\alpha_{p^m}) \cong \alpha_{p^m}^{n-m}\times\BG_a^m.$$\\
\end{proof}

\begin{rem}
\label{rem6}
\begin{itemize}
\item[1)] In both cases, any $n$-tuple $(a_1,\dotsc,a_n)$ with
$a_i\in \BG_a(R)$ (in the second case, the first $m$ entries are in
fact in $\alpha_{p^m}^{n-m}(R)$) defines a Hopf algebra homomorphism
$\phi:R[X]/(X^{p^m})\to R[X]/(X^{p^n})$ with
$$\phi(x)=a_1x+a_px^p+a_{p^2}x^{p^2}+\dotsb+a_{p^{n-1}}x^{p^{n-1}}.$$
It then follows that the homomorphism
$\alpha_{p^n,R}\to\alpha_{p^m,R}$ corresponding to $\phi$, sends
$s\in\alpha_{p^n}(S)$ to
$$a_1s+a_ps^p+a_{p^2}s^{p^2}+\dotsb+a_{p^{n-1}}s^{p^{n-1}}\in\alpha_{p^m}(S)$$
for any $R$-algebra $S$.
\item[2)] The same arguments as in the first case of the example, show that for any positive integer $n$, there is
an isomorphism $\innHom(\alpha_{p^n},\BG_a)\cong\BG_a^n$. And as in
the first part of the remark, this isomorphism sends any $n$-tuple
$(a_1\dotsc,a_n)\in\BG_a(R)^n$ to the homomorphism
$\alpha_{p^n,R}\to\BG_{a,R}$ that sends an element
$s\in\alpha_{p^n}(S)$ to the element
$$a_1s+a_ps^p+a_{p^2}s^{p^2}+\dotsb+a_{p^{n-1}}s^{p^{n-1}}\in\BG_a(S)$$
for any $R$-algebra $S$.\xqed{\lozenge}
\end{itemize}
\end{rem}

Now, we can go further and show:

\begin{prop}
\label{ex4}
For any positive
integers $n_1,n_2,\dotsc,n_r$ we have an isomorphism
$$\innMult(\alpha_{p^{n_1}}\times\alpha_{p^{n_2}}\times\dotsb\times\alpha_{p^{n_r}},\BG_a)\cong\BG_a^{n_1n_2\dotsb
n_r}$$given by the following formula: An element
$$\overrightarrow{a}_r:=(a_{i_1,i_2,\dotsc,i_r})\in\BG_a(R)^{n_1n_2\dotsb
n_r}$$ corresponds to the homomorphism
$$\phi_{\overrightarrow{a}_r}:\alpha_{p^{n_1},R}\times\alpha_{p^{n_2},R}\times\dotsb\times\alpha_{p^{n_r},R}\to\BG_{a,R}$$
which sends an $r$-tuple
$$(s_1,s_2\dotsc,s_r)\in\alpha_{p^{n_1}}(S)\times\alpha_{p^{n_2}}(S)\times\dotsb\times\alpha_{p^{n_r}}(S)$$
to the element
$$\sum_{i_1,i_2,\dotsc,i_r}a_{i_1,i_2,\dotsc,i_r}s_1^{p^{i_1}}s_2^{p^{i_2}}\dotsb
s_r^{p^{i_r}}\in\BG_a(S),\quad i_j\in\{0,1,\dotsc,n_j-1\}\ \forall
1\leq j\leq r$$ for any $R$-algebra $S$.
\end{prop}

\begin{proof}[\textsc{Proof}. ]
We recall that for commutative group schemes $G,H_1$ and $H_2$ we
have
$$\Hom(G,H_1\times H_2)\cong\Hom(G,H_1)\times\Hom(G,H_2)$$ and
therefore, we also have the underlined version
$$\innHom(G,H_1\times H_2)\cong\innHom(G,H_1)\times\innHom(G,H_2).$$
We show the statement by induction on $r$. If $r=1$, then this is
exactly Remark \ref{rem6} point 2). Suppose that $r>1$ and the
statement is true for $r-1$. Let us fix a $k$-algebra $R$, an
$R$-algebra $S$ and an $S$-algebra $T$ for the rest of the proof. We
have
$$\innMult(\alpha_{p^{n_1}}\times\alpha_{p^{n_2}}\times\dotsb\times\alpha_{p^{n_r}},\BG_a)
\cong\innHom(\alpha_{p^{n_1}},\innMult(\alpha_{p^{n_2}}\times\dotsb\times\alpha_{p^{n_r}},\BG_a))$$
by Proposition \ref{prop3}. By the induction hypothesis, we have
$$\innMult(\alpha_{p^{n_2}}\times\dotsb\times\alpha_{p^{n_r}},\BG_a)\cong\BG_a^{n_2\dotsb
n_r}$$ and under this isomorphism an element
$\overrightarrow{a}_{r-1}=(a_{i_2,\dotsc,i_r})\in\BG_a(R)^{n_2\dotsb
n_r}$ is sent to the homomorphism
$\phi_{\overrightarrow{a}_{r-1}}:\alpha_{p^{n_2},R}\times\dotsb\times\alpha_{p^{n_r},R}\to\BG_{a,R}$
defined above. Combining this isomorphism with the last one, we
obtain:
$$\innMult(\alpha_{p^{n_1}}\times\alpha_{p^{n_2}}\times\dotsb\times\alpha_{p^{n_r}},\BG_a)
\cong\innHom(\alpha_{p^{n_1}},\BG_a^{n_2\dotsb
n_r})\cong\innHom(\alpha_{p^{n_1}},\BG_a)^{n_2\dotsb n_r}.$$ By
Remark \ref{rem6} 2), $\innHom(\alpha_{p^{n_1}},\BG_a)^{n_2\dotsb
n_r}\cong \BG_a^{n_1n_2\dots n_r}$. Now we consider the image of an
element
$\overrightarrow{a}_r=(a_{i_1,i_2,\dotsc,i_r})\in\BG_a(R)^{n_1n_2\dotsb
n_r}$ under these isomorphisms. The isomorphism
$(\BG_a^{n_1})^{n_2\dotsb n_r}\cong \BG_a^{n_1\dotsb n_r}$ (it is in
fact a rearranging of entries) maps this element to
$(A_{i_2,\dotsc,i_r})$ where each
$$A_{i_2,\dotsc,i_r}=(a_{0,i_2,\dotsc,i_r},a_{1,i_2,\dotsc,i_r},\dotsc,a_{n_1-1,i_2,\dotsc,i_r})$$
is a vector in $\BG_a(R)^{n_1}$. Under the isomorphism
$\innHom(\alpha_{p^{n_1}},\BG_a)\cong \BG_a^{n_1}$, each vector
$A_{i_2,\dotsc,i_r}$ is sent to the homomorphism
$\phi_{i_2,\dotsc,i_r}:\alpha_{p^{n_1},R}\to\BG_{a,R}$ defined by
the vector $A_{i_2,\dotsc,i_r}$, as in the statement of Remark
\ref{rem6} 2), i.e.,
$$\phi_{i_2,\dotsc,i_r}(s)=a_{0,i_2,\dotsc,i_r}+a_{1,i_2,\dotsc,i_r}s^p+\dotsb+a_{n_1-1,i_2,\dotsc,i_r}s^{p^{n_1-1}}$$
for any $s\in \alpha_{p^{n_1}}(S)$. So we have an element
$(\phi_{i_2,\dotsc,i_r})\in\Hom(\alpha_{p^{n_1},R},\BG_{a,R})^{n_2\dotsb
n_r}$. Under the isomorphism
$\Hom(\alpha_{p^{n_1},R},\BG_{a,R}^{n_2\dotsb n_r})\cong
\Hom(\alpha_{p^{n_1},R},\BG_{a,R})^{n_2\dotsb n_r}$, this element
goes to $$\phi:\alpha_{p^{n_1},R}\to\BG_{a,R}^{n_2\dotsb n_r},\quad
s\mapsto (\phi_{i_2,\dotsc,i_r}(s))\in\BG_a(S)^{n_2\dotsb n_r}$$
where $s$ is an element of $\alpha_{p^{n_1}}(S)$. The element
$(\phi_{i_2,\dotsc,i_r}(s))$ corresponds by the induction hypothesis
to the multilinear morphism
$\phi':\alpha_{p^{n_2},S}\times\dotsb\times\alpha_{p^{n_r},S}\to\BG_{a,S}$
which sends an $r$-tuple
$(t_2,\dotsc,t_r)\in\alpha_{p^{n_2}}(T)\times\dotsb\times\alpha_{p^{n_r}}(T)$
to the element
$\sum_{i_2,\dotsc,i_r}\phi_{i_2,\dotsc,i_r}(s)t_2^{p^{i_2}}\dotsb
t_r^{p^{i_r}}=$
$$\sum_{i_2,\dotsc,i_r}\sum_{i_1=0}^{n_1-1}a_{i_1,i_2\dotsc,i_r}s^{p^{i_1}}t_2^{p^{i_2}}\dotsb
t_r^{p^{i_r}}=\sum_{i_1,\dotsc,i_r}a_{i_1,i_2\dotsc,i_r}s^{p^{n_1}}t_2^{p^{i_2}}\dotsb
t_r^{p^{i_r}}.$$ It follows from the isomorphism
$$\innMult(\alpha_{p^{n_1}}\times\alpha_{p^{n_2}}\times\dotsb\times\alpha_{p^{n_r}},\BG_a)
\cong\innHom(\alpha_{p^{n_1}},\innMult(\alpha_{p^{n_2}}\times\dotsb\times\alpha_{p^{n_r}},\BG_a))$$
that $\phi'$ is sent to
$\phi_{\overrightarrow{a}_r}:\alpha_{p^{n_1},R}\times\alpha_{p^{n_2},R}\times\dotsb\times\alpha_{p^{n_r},R}\to\BG_{a,R}$
as in the statement of the example. The proof is thus achieved.
\end{proof}

\begin{rem}
\label{rem7}
\begin{itemize}
\item[1)] It is easy to check that the corresponding Hopf algebra
homomorphism defined by
$\phi_{\overrightarrow{a}_r}:\alpha_{p^{n_1},R}\times\alpha_{p^{n_2},R}\times\dotsb\times\alpha_{p^{n_r},R}\to\BG_{a,R}$
is the $R$-homomorphism $$\phi^{\sharp}:R[X]\to
R[Y_1]/(Y_1^{p^{n_1}})\otimes\dotsb\otimes R[Y_r]/(Y_r^{p^{n_r}})$$
that sends $X$ to
$\sum_{i_1,\dotsc,i_r}a_{i_1\dotsc,i_r}y_1^{p^{i_1}}\otimes\dotsb\otimes
y_r^{p^{i_r}}$ where $y_j$ is the is the image of $Y_j$ in
$R[Y_j]/(Y_j^{p^{n_j}}).$
\item[2)] With the same methods as in the proof of Proposition \ref{ex4} and the
second part of Proposition \ref{ex3}, one can determine the group scheme
$$\innMult(\alpha_{p^{n_1}}\times\alpha_{p^{n_2}}\times\dotsb\times\alpha_{p^{n_r}},\alpha_{p^l}).$$
It would obviously depend on $l,r$ and $n_i$'s. The formula for the
general case (arbitrary $l,r$ and different $n_i$'s) is rather
complicated and we would not give it here, but for $l=1, r=n$ and
$n_1=n_2\dotsb=n_n=n$ we have
$$\innMult(\alpha_{p^n}^n,\alpha_p)\cong\alpha_p^{(n-1)^n}\times\BG_a^{(n-1)^n}.$$
\end{itemize}
\xqed{\lozenge}\end{rem}

We can use this example in order to calculate other interesting
groups of multilinear morphisms.

\begin{prop}
\label{ex5}
\emph{Let $k$ be a field of characteristic $p$. We have
isomorphisms:
\begin{itemize}
\item $\innSym(\alpha_{p^n}^r,\BG_a)\cong\BG_a^{\binom{n+r-1}{r-1}}$
\item $\innAlt(\alpha_{p^n}^r,\BG_a)\cong \BG_a^{\binom{n}{r}}$ if
$p>2$,
\end{itemize}
with the convention that $\binom{a}{b}=0$ whenever $b > a$.}
\end{prop}

\begin{proof}[\textsc{Proof}. ]
Let $R$ be a $k$-algebra and $$\phi:\alpha_{p^n,R}^r\to\BG_{a,R}$$
an element of $\Mult(\alpha_{p^n,R}^r,\BG_{a,R})$, which is
isomorphic to $\BG_a(R)^{n^r}$ by Proposition \ref{ex4}. So there exists an
element $(a_{i_1,\dotsc,i_r})\in\BG_a(R)^{n^r}$ that corresponds in
the way explained in that proposition to $\phi$. By the the first part
of Remark \ref{rem7}, the $R$-Hopf algebra homomorphism
corresponding to $\phi$ is the homomorphism
$$\phi^{\sharp}:R[X]\to R[Y]/(Y^{p^n})\otimes\dotsb\otimes
R[Y]/(Y^{p^n})$$ that sends $X$ to
$\sum_{i_1,\dotsc,i_r}a_{i_1\dotsc,i_r}y^{p^{i_1}}\otimes\dotsb\otimes
y^{p^{i_r}}$. The action of the symmetric group $S_r$ on
$\alpha_{p^n}^r$ and therefore on its representing Hopf algebra
$$\underbrace{R[Y]/(Y^{p^n})\otimes\dotsb\otimes R[Y]/(Y^{p^n})}_{r \text{\ times}}$$ permutes
$y^{p^{i_j}}$'s, i.e., if $\sigma\in S_r$, then
$\sigma(y^{p^{i_1}}\otimes\dotsb\otimes
y^{p^{i_r}})=y^{p^{i_{\sigma{1}}}}\otimes\dotsb\otimes
y^{p^{i_{\sigma{r}}}}$. We have thus,
$$\sigma(\sum_{i_1,\dotsc,i_r}a_{i_1\dotsc,i_r}y^{p^{i_1}}\otimes\dotsb\otimes
y^{p^{i_r}})=\sum_{i_1,\dotsc,i_r}a_{i_{\sigma{1}}\dotsc,i_{\sigma{r}}}y^{p^{i_1}}\otimes\dotsb\otimes
y^{p^{i_r}}$$
\begin{itemize}
\item $\phi$ is symmetric in $\alpha_{p^n}^r$ if and only if  $\phi^{\sharp}$ is
symmetric in the sense that it is invariant under composition with
any permutation, i.e., we must have $\phi^{\sharp}\circ
\sigma=\phi^{\sharp}$, or in other words,
$\sigma(\sum_{i_1,\dotsc,i_r}a_{i_1\dotsc,i_r}y^{p^{i_1}}\otimes\dotsb\otimes
y^{p^{i_r}})=\sum_{i_1,\dotsc,i_r}a_{i_1\dotsc,i_r}y^{p^{i_1}}\otimes\dotsb\otimes
y^{p^{i_r}}$, for all permutations $\sigma\in S_r$. But the elements
$y^{p^{i_1}}\otimes\dotsb\otimes y^{p^{i_r}}$ are linearly
independent over $R$. It follows then that
$a_{i_1\dotsc,i_r}=a_{i_{\sigma{1}}\dotsc,i_{\sigma{r}}}$, for all
permutations $\sigma\in S_r$. This is the only condition on
$a_{i_1,\dotsc,i_r}$ for the homomorphism $\phi$ to be symmetric. So
the number of different classes of $a_{i_1,\dotsc,i_r}$'s under the
action of $S_r$ is equal to the number of sequences of indices
$i_1\leq i_2\leq\dotsb\leq i_r$ with $0\leq i_j\leq n-1$, because
with the action of $S_r$ we can reorder the indices in this way.
This number is $\binom{n+r-1}{r-1}$. Hence, we have
$\Sym(\alpha_{p^n,R}^r,\BG_{a,R})\cong\BG_a(R)^{\binom{n+r-1}{r-1}}$,
which implies that
$\innSym(\alpha_{p^n}^r,\BG_a)\cong\BG_a^{\binom{n+r-1}{r-1}}$.
\item $\phi$ is alternating if and only if it is antisymmetric (since
the characteristic is odd). Then it is antisymmetric if and only if
$\phi^{\sharp}$ is antisymmetric. Arguing in the same way as above,
$\phi^{\sharp}$ is antisymmetric if and only if
$a_{i_{\sigma{1}}\dotsc,i_{\sigma{r}}}=\text{sgn}(\sigma)a_{i_1\dotsc,i_r}$.
In particular, every time two indices $i_{j_i}$ and $i_{j_2}$ are
equal $a_{i_1,\dotsc,i_r}$ vanishes. Here one uses again the fact
that $p$ is odd, indeed on the one hand interchanging $y$'s in
$j_1^{\text{th}}$ and $j_2^{\text{th}}$ factor of
$(R[Y]/(Y^{p^n}))^{\otimes r}$ doesn't change the sign ($y$ appears
with the same power in these factors) and on the other hand it
changes the sign (it is antisymmetric) and since $p$ is odd the
coefficient $a_{i_1,\dotsc,i_r}$ should be zero. Therefore the
number of possible nonzero $a_{i_1,\dotsc,i_r}$'s, i.e., those with
no restriction, is equal to the number of sequences of indices
$i_1<\dotsb <i_r$ with $0\leq i_j \leq n-1$. If $r$ is greater than
$n$ then this number is zero, otherwise this number is
$\binom{n}{r}$, so with the convention mentioned above it is always
$\binom{n}{r}$, and we have consequently
$\Alt(\alpha_{p^n,R}^r,\BG_{a,R})\cong \BG_a(R)^{\binom{n}{r}}$. It
follows at once that $\innAlt(\alpha_{p^n}^r,\BG_a)\cong
\BG_a^{\binom{n}{r}}$.
\end{itemize}
\end{proof}

\section{Tensor product and related constructions}

\begin{dfn}
\label{dfn5}
Let $S$ be a scheme and $G_1\dotsc,G_r,G$ commutative group schemes over $S$. A multilinear morphism
$\phi:G_1\times\dotsb\times G_r\to G$, or by abuse of terminology,
the group scheme $G$, is called a \emph{tensor product of}
$G_1,\dotsc,G_r$ if, for all commutative group schemes $H$ over $S$,
the induced map
$$\Hom(G,H)\to\Mult(G_1\times\dotsb\times G_r,H),\quad
\psi\mapsto\psi\circ\phi,$$ is an isomorphism. If such $G$ and
$\phi$ exist, we write $G_1\otimes\dotsb\otimes G_r$ for $G$.
\xqed{\blacktriangle}\end{dfn}

\begin{rem}
\label{rem8}
\begin{itemize}
\item[1)] The defining universal property of the tensor product,
makes it unique up to unique isomorphism to the extent that it
exists, and if this is so, the tensor product is functorial and
right exact in all arguments.
\item[2)] According to Theorem 4.3 in \cite{Pink1}, if $S=\Spec k$ for a field $k$, and $G_1,\dotsc,G_r$ are finite over $S$, then
$G_1\otimes\dotsb\otimes G_r$ exists and is pro-finite over $S$,
i.e., it is an inverse limit of finite group schemes over $S$. It is
in fact the inverse limit $\invlim G_{\alpha}^*$ where $G_{\alpha}$
runs over all finite subgroup schemes of
$\innMult(G_1\times\dots\times G_r,\BG_m)$. Again, every time we use
the tensor product of group schemes, we will assume the hypotheses
in this theorem so that this tensor product exists.
\item[3)] One would expect that the construction of tensor product
commutes with the base change, i.e., $(G_1\otimes\dots\otimes
G_r)_T\cong G_{1,T}\otimes\dots\otimes G_{r,T}$. But this is not
true as the example $\alpha_{p,k}\otimes\alpha_{p,k}$ shows. Indeed,
we have for any field $L$ that
$\alpha_{p,L}\otimes\alpha_{p,L}\cong\invlim G_i^*$ where $G_i$ runs
over all finite subgroup schemes of $\BG_{a,L}$  and if $k'/k$ is a
transcendental field extension of characteristic $p$, then there are
finite subgroup schemes of $\BG_{a,k'}$ that do not lie in a finite
subgroup scheme defined over $k$, so the inverse limit over $k'$ is
taken over a much larger system than over $k$. But for finite field
extensions this problem does not occur.\xqed{\lozenge}
\end{itemize}
\end{rem}

\begin{dfn}
\label{dfn6}
Let $G$ be a commutative group scheme over a base
scheme $S$.
\begin{itemize}
\item[\emph{(i)}] A symmetric multilinear morphism $\phi:G^r\to G'$,
or by abuse of terminology, the group scheme $G'$, is called an
\emph{$r^{\text{th}}$ symmetric power} of $G$, if for all
commutative group schemes $H$ over $S$, the induced map
$$\Hom(G',H)\to \Sym(G^r,H),\quad \psi\mapsto \psi\circ\phi,$$ is an
isomorphism. If such $G'$ and $\phi$ exist, we write $S^rG$ for
$G'$.

\item[\emph{(ii)}] An alternating multilinear morphism $\phi:G^r\to G'$,
or by abuse of terminology, the group scheme $G'$, is called an
\emph{$r^{\text{th}}$ alternating power} of $G$, if for all
commutative group schemes $H$ over $S$, the induced map
$$\Hom(G',H)\to \Alt(G^r,H),\quad \psi\mapsto \psi\circ\phi,$$ is an
isomorphism. If such $G'$ and $\phi$ exist, we write $\Lambda^rG$
for $G'$.\xqed{\blacktriangle}
\end{itemize}
\end{dfn}

\begin{rem}
\label{rem9} Again, if $S^rG$ resp. $\Lambda^rG$ exists, it with the
multilinear morphism $G^r\to S^rG$, resp. $G^r\to \Lambda^rG$, is
unique up to unique isomorphism.
\xqed{\lozenge}\end{rem}

\begin{prop}
\label{prop13} Let $G$ be a commutative group scheme. If
$\Lambda^nG=0$ then we have $\Lambda^mG=0$ for all $m\geq n$.
\end{prop}

\begin{proof}[\textsc{Proof}. ]
We show that $\Lambda^{n+1}G=0$; the result follows immediately by
induction. Let $H$ be a commutative group scheme. By the definition,
we have an isomorphism $\Hom(\Lambda^{n+1}G,H)\cong\Alt(G^{n+1},H).$
Under the isomorphism
$$\Mult(G^{n+1},H)\cong\Mult(G^n,\innHom(G,H))$$ given in Lemma
\ref{lem1}, the image of the subgroup $\Alt(G^{n+1},H)$ lies in the
subgroup $\Alt(G^n,\innHom(G,H))$ of $\Mult(G^n,\innHom(G,H))$.
Again by definition, we have
$$\Mult(G^n,\innHom(G,H))\cong\Hom(\Lambda^nG,\innHom(G,H)).$$ The
latter group is trivial by hypothesis. Therefore, we have
$\Hom(\Lambda^{n+1}G,H)=0$ for all commutative group schemes $H$,
which implies that $\Lambda^{n+1}G=0$.
\end{proof}

In order to show the existence of symmetric and alternating powers
of a commutative group scheme $G$ under good conditions, we have to
make a digression on quotients of inverse limits and the notion of
largest quotient:

\paragraph{\normalsize{Quotients of inverse limits}}$$ $$

Let $G=\invlim G_{\alpha}$ be a filtered inverse limit of affine
commutative group schemes and $G\twoheadrightarrow H$ a quotient
morphism. Let $A_{\alpha}, A$ resp. $B$ be the Hopf algebras
representing the group schemes $G_{\alpha}, G$ resp. $H$. We have
$B\subset A$ and $A=\bigcup_{\alpha} A_{\alpha}$ where the union is
filtered; and consequently $B=\bigcup_{\alpha} B_{\alpha}$ where
$B_{\alpha}=B\cap A_{\alpha}$ and this union is filtered too. Using
the fact that $(A_{\alpha}\otimes_kA_{\alpha})\cap
(B\otimes_kB)=(A_{\alpha}\cap B)\otimes_k(A_{\alpha}\cap
B)=B_{\alpha}\otimes_kB_{\alpha}$, we see that $B_{\alpha}$ is in
fact a Hopf algebra. It follows then that $H=\invlim H_{\alpha}$
where $H_{\alpha}=\Spec B_{\alpha}$ is the group scheme associated
to $B_{\alpha}$.\\

\paragraph{\normalsize{Largest quotient}} $$  $$

Let $\Gamma$ be a finite group acting on an abelian group $G$. Then
the \emph{largest quotient} of $G$ where $\Gamma$ acts trivially is
a quotient $\pi: G\twoheadrightarrow \widetilde{G}$ with the
following universal property: given an abelian group $H$ with a
trivial $\Gamma$-action and a $\Gamma$-equivariant homomorphism
$\phi:G\rightarrow H$ there exists a unique homomorphism
$\widetilde{\phi}:\widetilde{G}\rightarrow H$ which makes the
following diagram commute

$$\xymatrix{
G\ar@{->>}[dr]_{\pi}\ar[rr]^{\phi}&&H\\
&\widetilde{G}\ar@{-->}[ur]_{\widetilde{\phi}}^{\exists !}.}$$

It is easy to see that the largest quotient is unique up to unique
isomorphism and if we write it as quotient of $G$ by a subgroup then
it is unique. One can verify easily that explicitly the largest
quotient is the cokernel of the homomorphism
$$(\prod_{\gamma\in\Gamma} G)\rightarrow G, \quad (g_{\gamma})\mapsto \sum_{\gamma\in\Gamma}(\gamma\cdot
g_{\gamma}-g_{\gamma}).$$\\

Now let $G$ be a commutative group scheme with a $\Gamma$-action. We
can define in the same fashion the largest quotient of $G$ where
$\Gamma$ acts trivially. Using the fact that this $\Gamma$-action
induces an action on every abelian group $G(X)$ for all schemes $X$,
one sees easily that the cokernel of the morphism
$$(\prod_{\gamma\in\Gamma} G)\rightarrow G, \quad (g_{\gamma})\mapsto \sum_{\gamma\in\Gamma}(\gamma\cdot
g_{\gamma}-g_{\gamma})$$ in the category of finite commutative group
schemes over the field $k$ is indeed the largest quotient of $G$ in
this category under the action of $\Gamma$.\\

Now we are ready to show the following theorem:

\begin{thm}
\label{thm1}
If $S=\Spec k$ for a field $k$, and $G$ is finite over
$S$, then $S^rG$ and $\Lambda^rG$ exist and are pro-finite over $S$.
\end{thm}

\begin{proof}[\textsc{Proof}. ]
Under the stated assumptions, the tensor product $G^{\otimes r}$ of
$r$ factors of $G$ exists and is pro-finite over $S$ by Theorem 4.3
in \cite{Pink1}. By its universal property the tensor product
inherits an action of the symmetric group $S_r$. It is now clear
that the largest quotient of $G^{\otimes r}$ where $S_r$ acts
trivially is a
symmetric power $S^rG$.\\

Now let $G'$ be the inverse limit of all finite quotients
$H_{\alpha}$ of $G^{\otimes r}$ with the property that the composite
morphism $\Delta^r_{ij}\hookrightarrow G^r\rightarrow G^{\otimes
r}\twoheadrightarrow H_{\alpha}$ is trivial for all $i,j$. We want
to show that $G'$ is an alternating power $\Lambda^r G$. Given a
morphism $\phi:G^{\otimes r}\to K$ with trivial composition
$\Delta^r_{ij}\into G^r\rightarrow G^{\otimes r}\xrarr {\phi} K$ for
all $i,j$, let $K'$ be the image of $\phi$, then since $K'\into K$
is a monomorphism, the composite $\Delta^r_{ij}\into G^r\to
G^{\otimes r}\onto K'$ is zero. According to what we have shown
about quotients of inverse limits and since $G^{\otimes r}$ is
pro-finite, its quotient $K'$ is pro-finite too. We can thus write
$K'=\invlim K_{\beta}$ for finite $K_{\beta}$'s. We have therefore a
unique morphism $G'=\invlim H_{\alpha}\to K'$ (since $G^{\otimes
r}\onto K'\onto K_{\beta}$ is a finite quotient of $G^{\otimes r}$
with trivial composition with $\Delta^r_{ij}\into G^r\to G^{\otimes
r}$, $K_{\beta}$ appears in the filtered system of the inverse limit
$\invlim H_{\alpha}$). It follows that $G'=\Lambda^r G$.
\end{proof}

\begin{prop}
\label{prop9} Let $G_1, G_2$ and $F$ be commutative group schemes
and $r_1,r_2$ two positive integers. We have a natural isomorphism
$$\Alt(G_1^{r_1}\times G_2^{r_2},F)\cong\Hom(\Lambda^{r_1}G_1\otimes \Lambda^{r_2}G_2,F).$$
\end{prop}

\begin{proof}[\textsc{Proof}. ]
Using Proposition \ref{prop11} we have a natural isomorphism
$$\Alt(G_1^{r_1}\times G_2^{r_2},F)\cong\Alt(G_1^{r_1},\innAlt(G_2^{r_2},F))$$and by
definition of $\Lambda^{r_1}G_1$ this is isomorphic to
$\Hom(\Lambda^{r_1}G_1,\innAlt(G_2^{r_2},F))$ which is again by
Proposition \ref{prop11} isomorphic to $\Alt(\Lambda^{r_1}G_1\times
G_2^{r_2},F)\cong$
$$\Alt(G_2^{r_2},\innHom(\Lambda^{r_1}G_1,F))\cong\Hom(\Lambda^{r_2}G_2,\innHom(\Lambda^{r_1}G_1,F))$$
$$\cong\Mult(\Lambda^{r_1}G_1\times\Lambda^{r_2}G_2,F)\cong\Hom(\Lambda^{r_1}G_1\otimes\Lambda^{r_2}G_2,F).$$
\end{proof}

\begin{rem}
\label{rem13}
\begin{itemize}
\item[1)] Let
$\lambda_{1,2}:G_1^{r_1}\times G_2^{r_2}\to\Lambda^{r_1}G_1\otimes
\Lambda^{r_2}G_2$ be the multilinear morphism in
$\Alt(G_1^{r_1}\times G_2^{r_2},\Lambda^{r_1}G_1\otimes
\Lambda^{r_2}G_2)$ that maps to the identity of
$\Lambda^{r_1}G_1\otimes \Lambda^{r_2}G_2$ by the isomorphism given
in the Proposition \ref{prop9}. Then, one can easily see that the
group scheme $\Lambda^{r_1}G_1\otimes\Lambda^{r_2}G_2$ has the
following universal property: Given any multilinear morphism
$\phi:G_1^{r_1}\times G_2^{r_2}\to F$ which is alternating in
$G_1^{r_1}$ and $G_2^{r_2}$, there exists a unique homomorphism
$\widetilde{\phi}:\Lambda^{r_1}G_1\otimes\Lambda^{r_2}G_2\to F$
making the following diagram commute:
$$\xymatrix{G_1^{r_1}\times
G_2^{r_2}\ar[rr]^{\phi}\ar[dr]_{\lambda_{1,2}}&&F\\
&\Lambda^{r_1}G_1\otimes\Lambda^{r_2}G_2.\ar[ur]_{\exists!\widetilde{\phi}}&
}$$
\item[2)] It is clear that we can generalize the Proposition
\ref{prop9}, i.e., if $G_1,\dots,G_n$ are commutative group schemes
and $r_1,\dots,r_n$ are positive integers, then there is a
multilinear morphism
$$\lambda_{1,\dots,n}:G_1^{r_1}\times\dots\times G_n^{r_n}\to\Lambda^{r_1}G_1\otimes\dots\otimes\Lambda^{r_n}G_n$$
alternating in each $G_i^{r_i}$ such that the homomorphism
$$\Hom(\Lambda^{r_1}G_1\otimes\dots\otimes\Lambda^{r_n}G_n,F)\to\Alt(G_1^{r_1}\times\dots\times G_n^{r_n},F)$$
$$\phi\mapsto\phi\circ\lambda_{1,\dots,n}$$is an isomorphism.\xqed{\lozenge}
\end{itemize}
\end{rem}

We know from the construction of the tensor product, which is given
in the proof of Theorem 4.3 in \cite{Pink1}, that
$G_1\otimes\dotsb\otimes G_r\cong\invlim G_{\alpha}^*$ where
$G_{\alpha}$ runs through all finite subgroups of
$\innMult(G_1\times\dotsb\times G_r,\BG_m)$ and we know that if this
group scheme is isomorphic to another group scheme $H$, then the
corresponding inverse limits of it and $H$ are isomorphic too and we
deduce that the tensor product of $G_1,\, G_2,\dotsc,\, G_r$ is
uniquely determined up to unique isomorphism by
$(\psi,\innMult(G_1\times\dotsb\times G_r,\BG_m))$, where $\psi$ is
the universal multilinear morphism associated to $G_1\dotsc, G_r$.

Now suppose that $G_1=G_2=\dotsb=G_r$ and the universal multilinear
morphism $\psi:\innMult(G^r,\BG_m)\times G^r\to \BG_m$ is symmetric
(resp. alternating) in $G^r$. Then any multilinear morphism
$\phi:H\times G^r\to \BG_m$ is symmetric (resp. alternating) which
follows from the commutativity of the diagram
$$\xymatrix{
H\times G^r\ar[dr]_{\widetilde{\phi}\times \Id_{G^r}}\ar[rr]^{\phi}&&\BG_m\\
&\widetilde{G}\times G^r\ar[ur]_{\psi_{\widetilde{G}}}&}$$with the
notations of Proposition \ref{prop6}.

If $H$ is a finite commutative group scheme and $\phi:G^r\to H$ is
multilinear, then $\phi$ is symmetric (resp. alternating) if and
only if the corresponding multilinear morphism $H^*\times
G^r\to\BG_m$ given by Cartier duality and Lemma \ref{lem1} is
symmetric (resp. alternating) but as we have seen, in our situation
any such multilinear morphism is symmetric (resp. alternating) and
thus
$$\Hom(G^{\otimes r},H)\cong\Mult(G^r,H)\cong\Sym(G^r,H)\cong\Hom(S^rG,H).$$ If $H$ is of finite type,
then any multilinear morphism $\phi:G^r\to H$ factors through a
finite subgroup $H'$ of $H$, i.e., we have a commutative diagram:
$$\xymatrix{
G^r\ar[rr]^{\phi}\ar[dr]_{\phi'}&&H\\
&H'\ar@{^{ (}->}[ur].&}$$ As we proved above, $\phi'$ is symmetric
(resp. alternating), hence $\phi$ is symmetric (resp. alternating)
as well and we have again
$$\Hom(G^{\otimes r},H)\cong\Mult(G^r,H)\cong\Sym(G^r,H)\cong\Hom(S^rG,H).$$

Now if we are in the general case, i.e., $H$ is any commutative
group scheme, then we can write it as an inverse limit of
commutative group schemes of finite type, say $H=\invlim
H_{\alpha}$. We have
$$\Hom(G^{\otimes r},H)\cong\Hom(G^{\otimes r},\invlim
H_{\alpha})\cong\invlim\Hom(G^{\otimes r},H_{\alpha})\cong$$
$$\invlim\Hom(S^rG,H_{\alpha})\cong\Hom(S^rG,\invlim
H_{\alpha})\cong\Hom(S^rG,H).$$ We have thus in any case that
$\Hom(G^{\otimes r},H)\cong\Hom(S^rG,H)$ and it implies that
$G^{\otimes r}\cong S^rG$. The same arguments hold in the
alternating case and we have in this case that $G^{\otimes
r}\cong\Lambda^rG$.\\

\begin{ex}
\label{ex6}Here we give a concrete example and calculate the
tensor product $\alpha_p\otimes\alpha_p$, the symmetric power
$S^2\alpha_p$ and the alternating power $\Lambda^2 \alpha_p$ over
$S=\Spec k$ for a field $k$ of characteristic $p>0$. In order to do
this, we try to find the universal group scheme associated to
$\alpha_p,\alpha_p$ (see Definition \ref{dfn7}) that we denote by
$\widetilde{\alpha_p}^2$ in this example and the universal
multilinear morphism
$\psi_{\widetilde{\alpha_p}^2}:\widetilde{\alpha_p}^2\times\alpha_p\times\alpha_p\to\BG_m$.
The universal group $\widetilde{\alpha_p}^2$ is the group
$\innMult(\alpha_p^2,\BG_m)\overset{1.10}{\cong}\innHom(\alpha_p,\innHom(\alpha_p,\BG_m))$,
and we have isomorphisms
$$\innHom(\alpha_p,\BG_m)\cong\alpha_p^*,\
\alpha_p^*\cong\alpha_p \ \ \text{and}\ \
\BG_a\cong\innHom(\alpha_p,\alpha_p).\qquad (\blacktriangle)$$ It
follows then that
$$\widetilde{\alpha_p}^2\cong\innHom(\alpha_p,\innHom(\alpha_p,\BG_m))
\cong\innHom(\alpha_p,\alpha_p^*)\cong\innHom(\alpha_p,\alpha_p)\cong\BG_a.\qquad
(\bigstar)$$ Using these isomorphisms and those stated at the
beginning  of this section we can write:
$$
\Mult(\BG_a\times\alpha_p\times\alpha_p,\BG_m)\cong\Mult(\BG_a\times\alpha_p,\innHom(\alpha_p,\BG_m))$$
$$\cong\Hom(\BG_a,\innHom(\alpha_p,\innHom(\alpha_p,\BG_m)))\cong\Hom(\BG_a,\widetilde{\alpha_p}^2).$$
If we identify $\widetilde{\alpha_p}^2$ with $\BG_a$ via the
isomorphism $\,(\bigstar)\,$, then in order to find the universal
multilinear morphism
$\phi:\BG_a\times\alpha_p\times\alpha_p\to\BG_m$ we have to chase
through these isomorphism and find the element of
$\Mult(\BG_a\times\alpha_p\times\alpha_p,\BG_m)$ corresponding to
the inverse of the isomorphism $\,(\bigstar)\,$ in
$\Hom(\BG_a,\widetilde{\alpha_p}^2)$.\\

Before doing this, we explain the isomorphisms
$\,(\blacktriangle)\,$. The isomorphism
$$\BG_a\cong\innHom(\alpha_p,\alpha_p)$$ is given for any $k$-algebra
$R$, by the morphism $r\mapsto \lambda_r$ where,
$\lambda_r:\alpha_{p,R}\to\alpha_{p,R}$ is defined by
$$\lambda_r(S):\alpha_p(S)\to\alpha_p(S)\quad s\mapsto r\cdot s.$$
The isomorphism $\innHom(\alpha_p,\BG_m)\cong\alpha_p^*$ is a
general fact about finite commutative group schemes and we explain
it in the case where $G=\Spec A$ is a finite affine commutative
group scheme over $k$. Given $f\in G^*(R)$ for a $k$-algebra $R$, by
definition, $f$ is a $k$-algebra homomorphism from $A^*$, the dual
of $A$, to the $k$-algebra $R$. So $f$ defines an $R$-algebra
homomorphism from $A^*\otimes_kR\cong(A\otimes_kR)^*$ to $R$ which
we denote again by $f$, which is in particular $R$-linear and by
duality $((A\otimes_kR)^*)^*\cong A\otimes_kR$. It follows that
there is an element $a\in A\otimes_kR$ such that for any
$g\in(A\otimes_kR)^*$, we have $f(g)=g(a)$. The homomorphism $f$
being an $R$-algebra homomorphism is equivalent to $a$ being a
\emph{group-like element}, i.e., an element such that
$\Delta(a)=a\otimes a$ and $\epsilon(a)=1$. This shows that the
elements of $G^*(R)$ are in bijection with group-like elements of
$A\otimes_kR$. But any such element defines in a unique way a
homomorphism $\theta_{R,a}:G_R \to \BG_{m,R}$ as follows: given an
$R$-algebra $S$ and an element $\psi\in G(S)$, i.e., a $k$-algebra
homomorphism $A\to S$, then $\theta_{R,a}(\psi)\in\BG_m(S)$ is the
composite
$$k[x,x^{-1}]\arrover{i_a} A\otimes_kR\arrover{\psi\otimes
1}S\otimes_kR\arrover{m}S$$ where $i_a(x)=a$ and $m(s\otimes
r)=r\cdot s$. Hence we have the isomorphism
$G^*\cong\innHom(G,\BG_m).$\\

Now we explain the isomorphism $\alpha_p\cong\alpha_p^*$. We
first give the isomorphism of Hopf algebras, then the isomorphism
between the group schemes with $\alpha_P^*$ regarded as
$\innHom(\alpha_p,\BG_m)$ as explained above. The Hopf algebras of
$\alpha_p$ and $\alpha_p^*$ are respectively $k[Y]/(Y^p)$ and
$(k[Y]/(Y^p))^*$. Elements $1, y,\dotsc, y^{p-1}$, the
images of $1, Y,\dotsc, Y^{p-1}$ in $k[Y]/(Y^p)$, form a
$k$-basis of $k[Y]/(Y^p)$ and denote by $\xi_0, \xi_1, \xi_2,\dotsc,
\xi_{p-1}$ the dual basis of $(k[Y]/(Y^p))^*$. A direct calculation
shows that the morphism sending $\xi_i$ to $\frac{1}{i!}\cdot y^i$
gives a Hopf algebra isomorphism between $(k[Y]/(Y^p))^*$ and
$k[Y]/(Y^p)$. This isomorphism defines for any $k$-algebra $R$ an
isomorphism of abelian groups $\alpha_p(R)\to \alpha_p^*(R)$ as
follows: an element $r\in \alpha_p(R)$ defines a $k$-algebra
homomorphism $k[Y]/(Y^p)\to R$ sending $y$ to $r$; we have thus a
$k$-algebra homomorphism $(k[Y]/(Y^p))^*\to R$ sending $\xi_i$ to
$\frac{1}{i!}r^i$ and it gives canonically an $R$-algebra
homomorphism
$\gamma_r:(k[Y]/(Y^p)\otimes_kR)^*\cong(k[Y]/(Y^p))^*\otimes_kR\to
R$, which sends $\xi_i\otimes1$ to $\frac{1}{i!}r^i$. Suppose that
$\gamma_r$ corresponds to the group-like element
$\sum_{i=0}^{p-1}y^i\otimes u_i\in k[Y]/(Y^p)\otimes_kR$. So by
definition, we have for any element $g$ of
$(k[Y]/(Y^p)\otimes_kR)^*$ that
$\gamma_r(g)=g(\sum_{i=0}^{p-1}y^i\otimes u_i)$. In particular
taking $g=\xi_j\otimes1$ and we obtain:
$$\frac{1}{j!}r^j=\gamma_r(\xi_j\otimes1)=(\xi_j\otimes1)(\sum_{i=0}^{p-1}y^i\otimes
u_i)=\sum_{i=0}^{p-1}\xi_j(y^i)u_i=u_j.$$ We deduce that the element
$r\in \alpha_p(R)$ corresponds to the group-like element
$U_r:=\sum_{i=0}^{p-1}\frac{1}{i!}\cdot y^i\otimes r^i$, which
itself corresponds to the morphism $\theta_{R,r}:\alpha_{p,R}\to
\BG_{m,R}$ defined for any $R$-algebra $S$ by
$\theta_{R,r}(S):\alpha_p(S)\to \BG_m(S)$ sending an element $s\in
\alpha_p(S)$ to the composite
$$k[x,x^{-1}]\arrover{i_{U_r}}k[Y]/(Y^p)\otimes_kR\arrover{\psi_s\otimes1}S\otimes_kR\arrover{m}S$$
where $\psi_s(y)=s$. The image of $x$ via this composite is
$\sum_{i=0}^{p-1}\frac{1}{i!}(rs)^i$. Thus, regarding $\BG_m(S)$ as
a subset of $S$, i.e., the group of invertible elements, this
$k$-algebra homomorphism is the element
$\sum_{i=0}^{p-1}\frac{1}{i!}(rs)^i$.\\

Now, we can proceed to find the desired multilinear morphism
$\BG_a\times\alpha_p\times\alpha_p\to \BG_m$. From the above
arguments, it is clear that the isomorphism
$\phi:\BG_a\cong\widetilde{\alpha_p}^2$ is given for any $k$-algebra
$R$, by the morphism $$\phi_R:\BG_a(R)\to
\Hom_R(\alpha_{p,R},\alpha_{p,R}^*),\quad r\mapsto \phi_{R,r}$$
where $\phi_{R,r}:\alpha_{p,R}\to \alpha_{p,R}^*$ is defined as
follows: if $S$ is an $R$-algebra, then
$\phi_{R,r}(S):\alpha_p(S)\to \alpha_p^*(S)$ sends an element $s\in
\alpha_p(S)$ to the group-like element
$\sum_{i=0}^{p-1}\frac{1}{i!}(rs)^i$ or in other words, to the
element $\phi_{R,r,s}$ in $\Hom_S(\alpha_{p,S},\BG_{m,S})$ which
sends an element $t\in \alpha_p(T)$ for an $S$-algebra $T$ to the
element $\sum_{i=0}^{p-1}\frac{1}{i!}(rst)^i$. Under the isomorphism
$$\Hom(\BG_a,\innHom(\alpha_p,\innHom(\alpha_p,\BG_m)))\cong
\Mult(\BG_a\times\alpha_p,\innHom(\alpha_p,\BG_m))$$ $\phi$ is
mapped to the multilinear morphism
$\widetilde{\phi}:\BG_a\times\alpha_p\to \innHom(\alpha_p,\BG_m)$
that sends the element $(r,s)\in \BG_a(R)\times\alpha_p(R)$ to
$\phi_{R,r,s}\in \Hom_R(\alpha_{p,R},\BG_{m,R})$ for any $k$-algebra
$R$. And under the isomorphism
$$\Mult(\BG_a\times\alpha_p,\innHom(\alpha_p,\BG_m))\cong\Mult(\BG_a\times\alpha_p\times\alpha_p,\BG_m)$$
$\widetilde{\phi}$ is sent to the multilinear morphism
$\widehat{\phi}:\BG_a\times\alpha_p\times\alpha_p\to \BG_m$ which
maps the triple $(r,s,t)\in
\BG_a(R)\times\alpha_p(R)\times\alpha_p(R)$ to the element
$\sum_{i=0}^{p-1}\frac{1}{i!}(rst)^i\in\BG_m(R)$ for any $k$-algebra
$R$. The morphism $\widehat{\phi}$ is our universal multilinear
morphism. It is clearly symmetric in the second and third arguments
and it follows from preceding discussion that we have
$S^2\alpha_p\cong\alpha_p\otimes\alpha_p$. Therefore, any
multilinear morphism $\alpha_p\times\alpha_p\to H$ to any
commutative group scheme $H$ is symmetric and if the characteristic
$p$ is $2$, then it is automatically alternating and we have that
$\Lambda^2\alpha_p\cong S^2\alpha_p\cong\alpha_p\otimes\alpha_p$. If
the characteristic is not $2$ and if this multilinear morphism is
alternating, then it is trivial and it follows that the alternating
group $\Lambda^2\alpha_p$ is trivial.\xqed{\blacksquare}
\end{ex}

\begin{ex}
\label{ex7}
By Proposition \ref{ex2}, the universal group
$\innMult(\alpha_p^n,\BG_m)$ associated to the $n$-fold tensor
product $\alpha_p\otimes\dotsb\otimes\alpha_p$ with $n\geq 2$ is
isomorphic to $\BG_a$. Then similar calculations show that the
universal multilinear morphism
$$\psi:\innMult(\alpha_p^n,\BG_m)\times\alpha_p^n\to \BG_m$$ is given for any $k$-algebra $R$, by the morphism
$$(s,r_1,\dotsc,r_n)\mapsto \sum_{i=0}^{p-1}\frac{(s\cdot r_1\dotsb r_n)^i}{i!} .$$
This morphism is clearly symmetric in $\alpha_p^n$ and we have
therefore
\begin{itemize}
\item $S^n\alpha_p\cong\alpha_p^{\otimes n}$
\item $\Lambda^n\alpha_p=0$ if $p\neq 2$ and
\item $\Lambda^n\alpha_p\cong\alpha_p^{\otimes n}$ if $p=2$.
\end{itemize}
From the construction of tensor products, we know that
$\alpha_p\otimes\dotsb\otimes\alpha_p\cong\invlim G_{\gamma}^*$,
where $G_{\gamma}$ runs through all finite subgroups of
$\innMult(\alpha_p^n,\BG_m)$. But the latter group is by Proposition
\ref{ex2} isomorphic to the additive group $\BG_a$. Thus, tensor
products $\alpha_p^{\otimes n}$ for any $n\geq2$ are isomorphic
which implies that the symmetric powers $S^n\alpha_p$ and the
alternating powers $\Lambda^n\alpha_p$ are also independent from $n$
for $n\geq2.$\\

This result together with Proposition \ref{prop5} imply that
for any commutative group scheme $F$ we have
\begin{itemize}
\item $\innSym(\alpha_p^n,F)=\innMult(\alpha_p^n,F)$ and
\item $\innAlt(\alpha_p^n,F)=0$ if $p>2$ and
\item $\innAlt(\alpha_p^n,F)=\innMult(\alpha_p^n,F)$ if $p=2$
\end{itemize}
\xqed{\blacksquare}\end{ex}

For the rest of this section, let $H$ be an arbitrary
commutative group scheme.\\

\textbf{Notation.} Let $G',G''$ be subgroup schemes of $G$ and $F$ a
commutative group scheme. By $\Alt(G'^r\times G''^s\times F^t,H)$ we
mean the group of multilinear morphisms that are alternating in
$G'^r, G''^s$ and $(G'\cap G'')^{r+s}$ and when we say that a
multilinear morphism $G'^r\times G''^s\times F^t\to H$ is
alternating, we mean that it belongs to the group $\Alt(G'^r\times
G''^s\times F^t,H)$. Likewise, we define the group
$\Alt(G_1^{r_1}\times \dots\times G_n^{r_n}\times F_1^{s_1}\times
\dots\times F_m^{s_m},H)$ with $G_i$ subgroup schemes of $G$ and
$F_j$'s arbitrary commutative group schemes.

\begin{lem}
\label{lem4}
Let $\pi:G\onto G''$ be an epimorphism and let
$\phi:G''^r\to H$ be a multilinear morphism such that the
composition $\phi\circ \pi^r:G^r\to H$ is alternating. Then $\phi$
is alternating as well.
\end{lem}

\begin{proof}[\textsc{Proof}]
The morphism $\pi$ induces a morphism $\Delta \pi:\Delta G\to\Delta
G''$ between diagonals and since the morphism $\pi$ is epimorphic,
the morphism $\Delta\pi$ is epimorphic too.

Similarly, we have an induced epimorphism between
$\Delta_{ij}^rG\subset G^r$ and $\Delta_{ij}^r G''\subset G''^r$ for
all $1\leq i < j\leq r$, which we denote by $\Delta_{ij}^r \pi$. In
order to show that $\phi$ is alternating, we must show that for any
$1\leq i<j\leq r$ the composition $\Delta_{ij}^rG''\into
G''^r\arrover{\phi}H$ is trivial. But we have a commutative diagram
$$\xymatrix{
\Delta^r_{ij}G\ar[r]^{\Delta\pi}\ar@{^{
(}->}[d]_{\iota}&\Delta^r_{ij}G''\ar@{^{ (}->}[d]^{\iota''}&\\
G^r\ar[r]_{\pi^r}&G''^r\ar[r]_{\phi}&H.}$$ Since the composite
$\phi\circ\pi^r$ is alternating, the composition
$\phi\circ\pi^r\circ\iota$ is trivial, and so is the composition
$\phi\circ\iota''\circ\Delta\pi$. The morphism $\Delta\pi$ is
epimorphic and it follows that $\phi\circ\iota''$ is trivial.
\end{proof}

\begin{rem}
\label{rem10}
Let $G'$ be a subgroup scheme of $G$ and
$\pi:G\onto G''$ an epimorphism. It can be shown in the same fashion
that if the composition of a multilinear morphism $G''^r\times
G^s\times G'^t\to H$ with the epimorphism
$\pi^r\times\Id_{G's}\times\Id_{G'^t}:G^r\times G^s\times G'^t\to
G''^r\times G^s\times G'^t$ is alternating, then this multilinear
morphism is also alternating.\xqed{\lozenge}\end{rem}

\begin{lem}
\label{lem5} Let $G_1\dots,G_r$ be commutative group schemes and
$\psi:G_1\times\dots\times G_r\to H$ a multilinear morphism. Assume
that for some $1\leq i\leq r$ we have a short exact sequence $0\to
G'_i\arrover{\iota} G_i\arrover{\pi} G''_i\to 0$. If the restriction
$\psi|_{G_1\times \dots\times G'_i\times\dots\times G_r}$ is zero,
then there is a unique multilinear morphism $\psi': G_1\times
\dots\times G''_i\times\dots\times G_r\to H$ such that
$\psi=\psi'\circ(\Id_{G_1}\times\dots\times\pi\times
\dots\times\Id_{G_r})$ with $\pi$ at the $i^{\text{th}}$ place.
\end{lem}

\begin{proof}[\textsc{Proof}]
By functoriality of the isomorphism in Proposition \ref{prop2} we
have a commutative diagram:
$$\xymatrix{
\Mult(G_1\times\dots\times
G_r,H)\ar[r]^{\cong\qquad\qquad\quad}\ar[d]^{\overline{\pi}}&\Hom(G''_i,\innMult(G_1\times\dots\times
\check{G_i}\times\dots\times
G_r,H))\ar[d]^{\pi^*}\\
\Mult(G_1\times\dots\times
G_r,H)\ar[d]^{\overline{\iota}}\ar[r]^{\cong\qquad\qquad\quad}&\Hom(G_i,\innMult(G_1\times\dots\times
\check{G_i}\times\dots\times
G_r,H))\ar[d]^{\iota^*}\\
\Mult(G_1\times\dots\times
G_r,H)\ar[r]^{\cong\qquad\qquad\quad}&\Hom(G'_i,\innMult(G_1\times\dots\times
\check{G_i}\times\dots\times G_r,H))}$$ where the indicated maps are
the obvious ones and $\check{G_i}$ means that this factor is
omitted. The right column is exact and $\pi^*$ is injective, because
the sequence $0\to G'_i\arrover{\iota} G_i\arrover{\pi} G''_i\to 0$
is exact and the functor $\Hom(-,F)$ is left exact for any
commutative group scheme $F$. Therefore, the left column is exact
too and $\overline{\pi}$ is injective. The morphism $\psi$ is an
element of $\Mult(G_1\times\dots\times G_r,H)$ which goes to zero
under the map $\overline{\iota}$ (restriction map). By exactness,
there is a unique multilinear morphism
$\psi'\in\Mult(G_1\times\dots\times G''_i\times\dots\times G_r,H)$
which is mapped to $\psi$ under $\overline{\pi}$. This proves the
lemma.
\end{proof}

\begin{lem}
\label{lem6} Let $G$ be a commutative group scheme, and let $0\to
G'\arrover{\iota} G\arrover{\pi} G''\to 0$ be a short exact
sequence. Then the restriction map
$$\Alt(G^m,H)\to \Mult(G'\times G^{m-1},H)$$
is injective, whenever $\Lambda^mG''=0$.
\end{lem}

\begin{proof}[\textsc{Proof}]
Let $\phi:G^m\to H$ be an alternating morphism and assume that the
restriction $\phi|_{G'\times G^{m-1}}$ is zero. We will show that
there is a multilinear morphism $\phi':G''^m\to H$ such that
$\phi=\phi'\circ\pi^m$. The result will then follow, since by Lemma
\ref{lem5} $\phi'$ is also alternating, it is so inside the group
$\Alt(G''^m,H)\cong\Hom(\Lambda^mG,H)$, which is trivial by the
hypothesis. It follows that $\phi'$ and consequently $\phi$ are
zero.

Note that since $\phi$ is alternating and the restriction
$\phi|_{G'\times G^{m-1}}$ is zero, the restrictions
$\phi|_{G^i\times G'\times G^{m-i-1}}$ are zero for any $1\leq i\leq
m-1$.

Put $\phi_0=\phi$. We show by induction on $0\leq i\leq m$ that
there is a multilinear morphism $\phi_i:G''^i\times G^{m-i}\to H$
such that $\phi=\phi_i\circ(\pi^i\times \Id_{G^{m-i}})$. This is
clear for $i=0$, so let $i>0$ and assume that we have $\phi_{i-1}$
with the stated property. Consider the following commutative diagram
$$\xymatrix{
G^{i-1}\times G'\times G^{m-i}\ar[d]_{\widehat{\pi}}\ar@{^{
(}->}[r]^{\rho}&G^{i-1}\times G\times
G^{m-i}\ar[d]^{\widetilde{\pi}}\ar[dr]^{\phi}&\\
G''^{i-1}\times G'\times G^{m-i}\ar@{^{
(}->}[r]^{\rho'}&G''^{i-1}\times G\times G^{m-i}\ar[r]^{\qquad \quad
\phi_{i-1}}&H}$$where
$\widehat{\pi}=\pi^{i-1}\times\Id_{G'}\times\Id_{G^{m-i}}$,
$\widetilde{\pi}=\pi^{i-1}\times\Id_G\times \Id_{G^{m-i}}$ and
$\rho,\rho'$ are the inclusion morphisms. We have by hypothesis,
$0=\phi\circ \rho=\phi_{i-1}\circ\widetilde{\pi}\circ \rho$ which
implies that $\phi_{i-1}\circ\rho'\circ\widehat{\pi}=0$. The
morphism $\widehat{\pi}$ is epimorphic and so
$\phi_{i-1}\circ\rho'$, the restriction of $\phi_{i-1}$, is zero. We
can therefore apply Lemma \ref{lem6}, so there is a multilinear
morphism $\phi_i:G''^i\times G^{m-i}\to H$ such that
$\phi_i=\phi_{i-1}\circ(\Id_{G''^{i-1}}\times\pi\times\Id_{G^{m-i}})$.
We have thus $\phi=\phi_i\circ(\pi^i\times \Id_{G^{m-i}})$.

Now put $i=m$, the statement says that there is a multilinear
morphism $\phi_m:G''^m\to H$ with $\phi=\phi_m\circ\pi^m$. This
$\phi_m$ is the required $\phi'$.
\end{proof}

\begin{rem}
\label{rem11}
\begin{itemize}
\item[1)]In Lemma \ref{lem6}, obviously the other restriction maps, i.e., restrictions
to $G^r\times G'\times G^{n-r-1}$ for $1\leq r\leq n-1$ are
injective too.
\item[2)]It is clear that the image of the
restriction map in Lemma \ref{lem6} lies inside the group
$\Alt(G'\times G^{m-1},H)$. We have thus the injection
$$\Alt(G^m,H)\into \Alt(G'\times G^{m-1},H).$$
\end{itemize}
\xqed{\lozenge}\end{rem}

\begin{thm}
\label{lem7} Assume that $0\to G'\arrover{\iota} G\arrover{\pi}
G''\to 0$ is a short exact sequence of commutative group schemes.
Let $m$ be a nonnegative integer and write $m=m'+m''$ with non
negative integers $m'$ and $m''$. Consider the diagram
$$\xymatrix{
\Alt(G^m,H)\ar[r]^{\rho\qquad}&\Alt(G'^{m'}\times G^{m''},H)\\
&\Alt(G'^{m'}\times G''^{m''},H)\ar@{^{ (}->}[u]_{\pi^*}}$$ where
$\rho$ is the restriction map.
\begin{itemize}
\item[\emph{(a)}] If $\Lambda^{m''+1}G''=0$, then $\rho$ is injective.
\item[\emph{(b)}] If $\Lambda^{m'+1}G'=0$, then $\rho$ factors through $\pi^*$.
\item[\emph{(c)}] If both conditions hold, then there is a natural epimorphism
$$\zeta:\Lambda^{m'}G'\otimes\Lambda^{m''}G''\onto\Lambda^mG.$$
\item[\emph{(d)}] If furthermore the sequence is split, then the epimorphism
$\zeta$ is an isomorphism.
\end{itemize}
\end{thm}

\begin{proof}[\textsc{Proof}]
If $m=0$ then $m'=0=m''$ and all statements are trivially true, so
assume $m>0$. We prove each point of the proposition separately.
\begin{itemize}
\item[(a)] Fix $m$. We show by induction on $0\leq m'\leq m$
that the restriction map gives an injective map
$$\Alt(G^m,H)\into \Alt(G^{m''},\innMult(G'^{m'},H)).$$

If $m'=0$ then $m''=m$, and $\rho$ is the identity map, so there is
nothing to show. So assume that $0<m'\leq m$ and that the statement
is true for $m'-1$ and $m''+1$ in place of $m'$ and $m''$. Then
$\Lambda^{m''+1}G''=0$ implies $\Lambda^{m''+2}G''=0$ by Proposition
\ref{prop13}; so by the induction hypothesis we have an injection
$$\Alt(G^m,H)\into \Alt(G^{m''+1},\innMult(G'^{m'-1},H)).$$ Since by hypothesis we have
$\Lambda^{m''+1}G''=0$ we can use Lemma \ref{lem6}, and we have thus
an injection
$$\Alt(G^{m''+1},\innMult(G'^{m'-1},H))\into\Alt(G^{m''}\times
G',\innMult(G'^{m'-1},H)).$$ The latter group is inside the group
$$\Alt(G^{m''},\innHom(G',\innMult(G'^{m'-1},H))).$$ By Proposition
\ref{prop3}, $\innHom(G',\innMult(G'^{m'-1},H))\cong
\innMult(G'^{m'},H)$.\\ Putting these together, we conclude that
there is an injection
$$\Alt(G^m,H)\into \Alt(G^{m''},\innMult(G'^{m'},H)).$$ Following through the above isomorphisms and inclusions, one verifies that this injection
is induced by the restriction map.

Under the isomorphism $$\Mult(G^{m''},\innMult(G'^{m'},H))\cong
\Mult(G^{m''}\times G'^{m'},H)$$ given by Proposition \ref{prop2},
the image of $\Alt(G^m,H)$ by $\iota$ lies inside the group
$\Alt(G^{m''}\times G'^{m'},H)$ and we can easily see that the
injection $$\Alt(G^m,H)\into \Alt(G^{m''}\times G'^{m'},H)$$ thus
obtained is given by the restriction map.

\item[(b)] Choose an alternating multilinear morphism $\phi:G^m\to H$ and write
$\phi_0$ for the restriction $\phi|_{G'^{m'}\times G^{m''}}$. For
any $0\leq j\leq m''-1$ the restriction of $\phi_0$ to the subgroup
scheme $G'^{m'}\times G^j\times G'\times G^{m''-j-1}$ belongs to the
group
$$\Alt(G'^{m'}\times G^j\times G'\times G^{m''-j-1},H)\into\Alt(G'^{m'+1},\innMult(G^{m''-1},H)).$$ The latter group is isomorphic to
$\Hom(\Lambda^{m'+1}G',\innMult(G^{m''-1},H))$, which is zero by
assumption. Therefore, the restriction $\phi_0|_{G'^{m'}\times
G^j\times G'\times G^{m''-j-1}}$ is zero.

Now we show by induction on $0\leq i\leq m''$, that there exists a
multilinear morphism $\phi_i:G''^i\times G^{m''-i}\times G'^{m'}\to
H$ such that the composition $$G^i\times G^{m''-i}\times
G'^{m'}\arrover{\overline{\pi}}G''^i\times G^{m''-i}\times
G'^{m'}\arrover{\phi_i}H$$ is $\phi_0$, where
$\overline{\pi}=\pi^i\times \Id_{G^{m''-i}}\times \Id_{G'^{m'}}$. If
$i=0$ then we have nothing to show, so let $i<m''$ and assume that
we have constructed $\phi_i$ with the desired property and we
construct $\phi_{i+1}$. Consider the following commutative diagram:
$$\xymatrix{
G^i\times G'\times G^{m''-i-1}\times
G'^{m'}\ar@{->>}[d]_{\widehat{\pi}}\ar@{^{ (}->}[r]^{\quad
\rho}&G^i\times
G^{m''-i}\times G'^{m'}\ar@{->>}[d]^{\overline{\pi}}\ar[dr]^{\phi_0}&\\
G''^i\times G'\times G^{m''-i-1}\times G'^{m'}\ar@{^{
(}->}[r]^{\quad \rho'}&G''^i\times G^{m''-i}\times
G'^{m'}\ar[r]^{\qquad \quad\phi_i}&H.}$$ As we have said above, the
restriction of $\phi_0$, $\phi_0\circ \rho$, is zero. By the
induction hypothesis, we have $\phi_0=\phi_i\circ\overline{\pi}$ and
therefore, $0=\phi_0\circ
\rho=\phi_i\circ\overline{\pi}\circ\rho=\phi_i\circ\rho'\circ\widehat{\pi}$.
The morphism $\widehat{\pi}$ being epimorphic, we conclude that the
restriction of $\phi_i$, i.e., $\phi_i\circ \rho'$ is zero. This
allows us to use Lemma \ref{lem5} in order to find a multilinear
morphism $\phi_{i+1}:G''^{i+1}\times G^{m''-i-1}\times G'^{m'}\to H$
such that
$\phi_i=\phi_{i+1}\circ(\Id_{G''^i}\times\pi\times\Id_{G^{m''-i-1}}\times\Id_{G'^{m'}})$.
It follows at once that $\phi_0=\phi_{i+1}\circ(\pi^{i+1}\times
\Id_{G^{m-i-2}}\times \Id_{G'})$.

Put $i=m''$, then the statement says that there is a multilinear
morphism $\phi_{m''}:G''^{m''}\times G'^{m'}\to H$ such that
$\phi_0=\phi_{m''}\circ(\pi^{m''}\times \Id_{G'^{m'}})$. Since
$\phi_0$ is alternating, by Remark \ref{rem10}, $\phi_{m''}$ is also
alternating.

\item[(c)] If both conditions hold, then by (a), $\rho$ is
injective and therefore the homomorphism
$\Alt(G^m,H)\to\Alt(G'^{m'}\times G''^{m''},H)$ defined in (b) is
injective as well. So we obtain
$$\Hom(\Lambda^mG,H)\cong\Alt(G^m,H)\into\Alt(G'^{m'}\times
G''^{m''},H)\overset{\ref{prop9}}{\cong}$$
$$\Hom(\Lambda^{m'}G'\otimes\Lambda^{m''}G'',H)$$ which is
natural, in other words we have a natural injection of functors
$$\tau:\Hom(\Lambda^mG,-)\into\Hom(\Lambda^{m'}G'\otimes\Lambda^{m''}G'',-).$$
It is a known fact that any natural transformation between such
functors is induced by a unique morphism
$\zeta:\Lambda^{m'}G'\otimes\Lambda^{m''}G''\to \Lambda^mG$, in
fact, this morphism is the image of the identity morphism of
$\Lambda^mG$ under this transformation. This means that for any
commutative group scheme $H$,
$\tau_H:\Hom(\Lambda^mG,H)\to\Hom(\Lambda^{m'}G'\otimes\Lambda^{m''}G'',H)$
sends a morphism $f:\Lambda^mG\to H$ to the morphism $f\circ \zeta$.
The injectivity of $\tau$ implies that $\zeta$ is epimorphic.

\item[(d)] 
Let $s:G''\to G$ be a section of $\pi$, i.e., $\pi\circ s=\Id_{G''}$
and $r:G\to G'$ the corresponding retraction of $\iota$, that is,
$r\circ\iota=\Id$ and that the short sequence $$0\to
G''\arrover{s}G\arrover{r}G'\to0$$ is exact. Then we show that the
map $\mu:\Alt(G^m,H)\to \Alt(G'^{m'}\times G''^{m''},H)$ whose
composition with $\pi^*$ is $\rho$ (given by (b)) is induced by the
inclusion $j:=\iota^{m'}\times s^{m''}:G'^{m'}\times G''^{m''}\into
G^m$. Indeed, given a morphism $f\in\Alt(G^m,H)$, we have
$\rho(f)=\pi^*(\mu(f))$, or in other words, $\mu(f)\circ
(\Id_{G'}^{m'}\times \pi^{m''})=f\circ (\iota^{m'}\times
\Id_G^{m''})$. Hence the following diagram is commutative
$$\xymatrix{
G'^{m'}\times G''^{m''}\ar@{^{ (}->}[d]_{\widetilde{s}}\ar@/_3pc/[dd]_{\Id}\ar@{^{ (}->}_{j\,}[dr]\ar[drrr]^{\mu(f)}&&\\
G'^{m'}\times G^{m''}\ar@{->>}[d]_{\widetilde{\pi}}\ar@{^{ (}->}[r]^{\quad i}&G^m\ar[rr]^{f\quad}&&H\\
G'^{m'}\times G''^{m''}\ar[urrr]_{\mu(f)}&&}$$ where $\widetilde{s},
\widetilde{\pi}$ and $i$ are respectively the morphisms
$\Id_{G'}^{m'}\times s^{m''}, \Id_{G'}^{m'}\times \pi^{m''}$ and the
inclusion $\iota^{m'}\times \Id_G^{m''}$. Consequently,
$\mu(f)=f\circ j$. This shows that $\mu$ is induced by $j$ as we
claimed.

Now define a morphism $\omega:G^m\to G'^{m'}\times G''^{m''}$ as
follows: for any $k$-algebra $R$ $\omega$ sends an element
$(g_1,\dots,g_m)\in G(R)^m$ to
$$\sum_{\sigma}\text{sgn}(\sigma,\tau)(r(g_{\sigma(1)}),\dots,r(g_{\sigma(m')}),\pi(g_{\tau(1)}),\dots,\pi(g_{\tau(m'')}))$$
where the sum runs over all length $m'$ subsequences
$\sigma=(\sigma(1),\dots,\sigma(m'))$ of $(1,2,\dots,m)$ with
complementary subsequences $\tau=(\tau(1),\dots,\tau(m''))$ and
$\text{sgn}(\sigma,\tau)$ is the signature of $(\sigma,\tau)$ as a
permutation of $m$ elements. This morphism induces a homomorphism
$$\omega^*:\Alt(G'^{m'}\times G''^{m''},H)\to \Mult(G^m,H)$$and it
is straightforward to see that in fact the image lies inside the
subgroup $\Alt(G^m,H)$. We also denote by $\omega^*$ the
homomorphism $\Alt(G'^{m'}\times G''^{m''},H)\to \Alt(G^m,H)$
obtained by restricting the codomain of $\omega^*$. Since the
composites $r\circ s$ and $\pi\circ\iota$ are trivial and
$r\circ\iota$ and $\pi\circ s$ are the identity morphisms, we see
that the composition $\omega\circ j$ is the identity morphism of
$G'^{m'}\times G''^{m''}$. Therefore the composite
$\mu\circ\omega^*$ is the identity homomorphism. Consequently, the
homomorphism $\mu:\Alt(G^m,H)\to\Alt(G'^{m'}\times G''^{m''},H)$ is
an epimorphism. We know from $(c)$ that it is a monomorphism, and
hence it is an isomorphism. We obtain thus
$$\Hom(\Lambda^mG,H)\cong\Alt(G^m,H)\isoto$$
$$\Alt(G'^{m'}\times
G''^{m''},H)\cong\Hom(\Lambda^{m'}G'\otimes\Lambda^{m''}G'',H).$$ As
we know, this homomorphism is induced by the morphism
$$\zeta:\Lambda^{m'}G'\otimes\Lambda^{m''}G''\to \Lambda^mG.$$
Since it is an isomorphism, the morphism $\zeta$ must be an
isomorphism as well.
\end{itemize}
\end{proof}

\begin{prop}
\label{prop7} Let $G$ be a local-local commutative group scheme of
order $p^n$ with $p$ an odd prime number. We have:
\begin{itemize}
\item[\emph{(a)}] $\Lambda^m G=0$ for all $m>n$.
\item[\emph{(b)}] $\Lambda^n G$ is a quotient  of $\alpha_p^{\otimes n}$.
\end{itemize}
\end{prop}

\begin{proof}[\textsc{Proof}]
We know that any subgroup of a local-local commutative group scheme
is again local-local. We can thus prove the proposition by induction
on $n$. If $n=1$, then $G$ is necessarily isomorphic to $\alpha_p$,
hence the equality $\Lambda^m\alpha_p=0$ follows from Example
\ref{ex7} and we have obviously $\Lambda^1\alpha_p=\alpha_p$, which
is a quotient of itself. So assume that $n>1$ and that the two
statements are true for positive integers less that $n$. Take a
proper subgroup scheme $G'$ of $G$ and let $G''$ be the quotient of
$G$ by $G'$, that is, we have a short exact sequence $0\to G'\to
G\to G''\to 0.$ We know that the order of commutative group schemes
is multiplicative, i.e., $|G|=|G'|\cdot|G''|$. So if $|G'|=p^{n'}$
and $|G''|=p^{n''}$, we have $n=n'+n''$. Take $m\geq n$, we can
write $m=m'+m''$ where $m''=n''$ and $m'=m-n''$ and we have $m'\geq
n'$. Since $G'$ is a proper subgroup scheme of $G$, we have $n'<n$
and so $n''<n$. Therefore, by the induction hypothesis we have
$\Lambda^{m'+1}G'=0=\Lambda^{m''+1}G''$. We can thus apply the third
point of Theorem \ref{lem7}, and we have an epimorphism
$$\zeta:\Lambda^{m'}G'\otimes\Lambda^{m''}G''\onto\Lambda^mG.$$

If $m>n$, then $m'>n'$ and we have $\Lambda^{m'}G'=0$ by the
induction hypothesis and so the tensor product
$\Lambda^{m'}G'\otimes\Lambda^{m''}G''$ vanishes. Since $\zeta$ is
epimorphic, we conclude that $\Lambda^mG=0$.

If $m=n$, then $m'=n'$. By the induction hypothesis, we have
epimorphisms $\xi':\alpha_p^{\otimes n'}\onto\Lambda^{n'}G'$ and
$\xi'':\alpha_p^{\otimes n''}\onto\Lambda^{n''}G''$. As we said in
Remark \ref{rem8}, the tensor product is right exact, and we have
thus an epimorphism
$$\alpha_p^{\otimes n}\cong\alpha_p^{\otimes
n'}\otimes\alpha_p^{\otimes
n''}\ontoover{\xi'\otimes\xi''}\Lambda^{n'}G'\otimes\Lambda^{n''}G''.$$
Composing this epimorphism with $\zeta$ we obtain the desired
epimorphism $$\alpha_p^{\otimes n}\onto\Lambda^nG.$$
\end{proof}

\begin{cor}
\label{cor1} Let $G$ and $H$ be local-local commutative group
schemes of order $p^n$ and $p^m$ respectively, with $p$ an odd prime
number. Then we have a natural isomorphism
$$\Lambda^{n+m}(G\oplus H)\cong\Lambda^nG\otimes\Lambda^mH.$$
\end{cor}

\begin{proof}[\textsc{Proof}]
By Proposition \ref{prop7}, we know that
$\Lambda^{n+1}G=0=\Lambda^{m+1}H$. The result follows at once from
the last point of Theorem \ref{lem7}.
\end{proof}

\begin{lem}
\label{lem10} Let $G$ be an affine commutative group scheme and
$H,F$ two finite subgroup schemes. Then there is a finite subgroup
scheme of $G$ that contains both $H$ and $F$.
\end{lem}

\begin{proof}[\textsc{Proof}]
Consider the homomorphism $$\mu:F\oplus H\to G,\quad
(f,h)\mapsto\mu(f,h):=f+h.$$ The image of this homomorphism contains
both $F$ and $H$ and its order is less than or equal to the order of
$F\oplus H$ which is finite. It is thus a finite subgroup scheme of
$G$.
\end{proof}

\begin{lem}
\label{quotinvlim} Let $I$ be a filtered system and for any $i\in
I$, $0\to N_i\arrover{f_i}G_i\arrover{g_i}Q_i\to0$ be a short exact
sequence of affine commutative group schemes. Then the short
sequence
$$0\longrightarrow\invlim_{i\in I}N_i\arrover{\invlim f_i}\invlim_{i\in I}G_i\arrover{\invlim
g_i}\invlim_{i\in I}Q_i\longrightarrow0$$ is exact, in other words,
taking filtered inverse limits is an exact functor.
\end{lem}

\begin{proof}[\textsc{Proof}. ]
To simplify the notation, we denote by $N, G$ and $Q$ the group
schemes $\invlim N_i, \invlim G_i$ and $\invlim Q_i$. Let $B_i, A_i$
and $C_i$ denote respectively the Hopf algebras associated to the
group schemes $N_i, G_i$ and $Q_i$. Let also $B, A$ and $C$ denote
respectively the Hopf algebras of the group schemes $N, G$ and $Q$,
i.e., $B=\bigcup_iB_i, A=\bigcup_iA_i$ and $C=\bigcup_iC_i$.
Finally, let $f_i:N_i\to G_i$ and $g_i:G_i\to Q_i$ be respectively
the morphism associated to the morphisms $f'_i:A_i\to B_i$ and
$g'_i:C_i\to A_i$.

The morphism $f:=\invlim f_i:N\to G$ is associated to the morphism
$\bigcup_i f'_i:\bigcup_iA_i\to\bigcup_iB_i$. Since $f_i'$ is
surjective for all $i\in I$, their union $\bigcup_if'_i$ is
surjective too and consequently $\invlim f_i$ is a monomorphism.
Similarly, since each $g'_i:C_i\to A_i$ is injective, the union
$\bigcup_ig'i:\bigcup_iC_i\to\bigcup_iA_i$ is injective and so the
morphism $g:=\invlim g_i$ is an epimorphism. It remains to show that
$Q$ is the quotient of $f:N\to G$.

Let $Q'$, with associated Hopf algebra $C'$, be the quotient of $f$,
i.e., we have a short exact sequence
$$0\to N\arrover{f}G\arrover{}Q'\to0.$$
We know that $C'$ equals
$$\{ x\in A\ | \ \Delta x\equiv x\otimes 1\  \text{mod}\ A\otimes J\},$$
the subspace of the regular representation where $N$ acts trivially.

We have for all $i\in I$ a commutative diagram
$$\xymatrix{
0\ar[r]&N\ar[r]\ar@{->>}[d]&G\ar@{->>}[d]\ar[r]&Q'\ar[r]&0\\
0\ar[r]&N_i\ar[r]&G_i\ar[r]&Q_i\ar[r]&0.}$$ We want to show that
$Q'=\invlim Q_i$. Since the composite $N\rightarrow G\rightarrow
G_i\rightarrow Q_i$ is trivial, there exists a unique morphism
$Q\rightarrow Q_i$ which makes the right square in the diagram
commute and since the composite $G\rightarrow G_i\rightarrow Q_i$ is
an epimorphism the induced morphism $Q'\to Q_i$ is an epimorphism
too. We can thus complete the diagram as follows

$$\xymatrix{
0\ar[r]&N\ar[r]\ar@{->>}[d]&G\ar@{->>}[d]\ar[r]&Q'\ar[r]\ar@{->>}[d]&0\\
0\ar[r]&N_i\ar[r]&G_i\ar[r]&Q_i\ar[r]&0.}$$

Writing $B_i=A_i/J_i$ and expressing the above commutative diagram
in terms of Hopf algebras we obtain for all $i\in I$ the following
commutative diagram

$$\xymatrix{
C_i\ar@{^{ (}->}[r]\ar@{^{ (}->}[d]&A_i\ar@{^{ (}->}[d]\ar@{->>}[r]&A_i/J_i\ar@{^{ (}->}[d]\\
C'\ar@{^{ (}->}[r]&A\ar@{->>}[r]&A/J. }$$

It follows then that $J_i\subset J$ and $J_i=J\cap A_i$. We can also
deduce from this that $J_i=A_i\cap J_j$ whenever $A_i\subset A_j$.
The inclusions $C_i\subset C'$ give an inclusion $\bigcup_i
C_i\subset C'$ and we should prove that this inclusion is in fact an
equality. Note that the union $\bigcup_i C_i$ is filtered, for if
given two indices $i$ and $j$ there is an index $l$ such that $A_i,
A_j\subset A_l$ and we have
$$C_l=\{ x\in A_l\ |\ \Delta x\equiv x\otimes 1\
\text{mod}\ A_l\otimes J_l \}$$ contains both $C_i$ and $C_j$. Now
let $x\in C'$, so $\Delta x-x\otimes1 \in A\otimes J$. Since
$\bigcup_iA_i=A$, $x$ is in some $A_i$ and so $\Delta x-x\otimes1
\in A_i\otimes A_i$ which implies that $\Delta x-x\otimes1 \in
A_i\otimes A_i\cap A\otimes J=A_i\otimes (A_i\cap J)=A_i\otimes J_i$
and therefore $x$ is in $C_i$. It follows at once
that $Q'=\invlim Q_i$.\\
\end{proof}

\begin{lem}
\label{lem11} Let $\BG_a^*$ denote the group scheme
$\underset{G_i\subset\BG_a}{\invlim}G_i^*$ where the limit is over
all finite subgroup schemes of $\BG_a$. Then we have a short exact
sequence
$$\BG_a^{*^{(p)}}\overset{V}{\to}\BG_a^*\to\alpha_p\to0$$where $V$ is
the Verschiebung.
\end{lem}

\begin{proof}[\textsc{Proof}]
A straightforward calculation shows that $$\alpha_p^{\otimes
n^{(p)}}\cong(\underset{G_i\subset \BG_a}{\invlim}
G_i^*)^{(p)}\cong\underset{G_i\subset \BG_a}{\invlim}
(G_i^{*^{(p)}})$$and that the Verschiebung $V:(\underset{G_i\subset
\BG_a}{\invlim} G_i^*)^{(p)}\to\underset{G_i\subset \BG_a}{\invlim}
G_i^*$ is the inverse limit of the Verschiebungen
$V_i:G_i^{*^{(p)}}\to G_i^*$. According to Lemma \ref{quotinvlim},
the cokernel of the inverse limit of the $V_i$ is the inverse limit
of the cokernel of the $V_i$, i.e., $\cokernel
\invlim_iV_i=\invlim_i\cokernel V_i$.

Now we have $G_i^{*^{(p)}}\cong G_i^{(p)^*}$ and
$V_i:G_i^*{^{(p)}}\to G_i^*$ is the dual of the Frobenius
$F_i:G_i\to G_i^{(p)}$. By duality, we have $\cokernel F_i^*\cong
(\kernel F_i)^*$. Putting these facts together, we obtain $\cokernel
V\cong \invlim (\kernel F_i)^*$. From Lemma \ref{lem10} we deduce
that the finite subgroup schemes of $\BG_a$ that contain $\alpha_p$
form a cofinal system and we can thus suppose that every $G_i$
contains $\alpha_p$. It follows that the kernel of the Frobenius
$F_i:G_i\to G_i^{(p)}$ is equal to $\alpha_p$. Hence
$$\cokernel V\cong \invlim (\kernel F_i)^*\cong\invlim
\alpha_p^*\cong \alpha_p^*\cong\alpha_p.$$
\end{proof}

\begin{lem}
\label{lem12}

Let $G_1,\dots,G_n$ and $H$ be finite commutative group schemes and
$\phi:G_1\times\dots\times G_n\to H$ a multilinear morphism. Then we
have a commutative diagram
$$\xymatrix{
G_1^{(p)}\times G_2^{(p)}\times\dots\times
G_n^{(p)}\ar[rr]^{\qquad\quad\phi^{(p)}}&&H^{(p)}\ar[dd]^{V_H}\\
G_1^{(p)}\times G_2\times\dots\times
G_n\ar[d]_{\widetilde{V}}\ar[u]^{\widetilde{F}}&&\\
G_1\times G_2\times\dots\times G_n\ar[rr]_{\qquad\phi}&&H}$$ where
$\widetilde{F}=\Id_{G_1^{(p)}}\times F_{G_2}\times\dots\times
F_{G_n}$ and $F_{G_i}:G_i\to G_i^{(p)}$ is the Frobenius of $G_i$
and $\widetilde{V}=V_{G_1}\times\Id_{G_2}\times
\dots\times\Id_{G_n}$, and $V_{G_1}$ and $V_H$ are the
Verschiebungen of $G_1$ and $H$.
\end{lem}

\begin{proof}[\textsc{Proof}]
Consider the following diagram
$$\xymatrix{
\Mult(G_1\times G_2\times\dots\times
G_n,H)\ar[r]_{\cong}^{\theta_1}\ar[d]_{(-)\circ
\widetilde{V}}&\Mult(G_2\times \dots\times G_n\times
H^*,G_1^*)\ar[d]^{F_{G_1^*}\circ(-)}\\
\Mult(G_1^{(p)}\times G_2\times\dots\times
G_n,H)\ar[r]_{\cong}^{\theta_2}&\Mult(G_2\times \dots\times
G_n\times
H^*,G_1^{(p)^*})\\
\Mult(G_1^{(p)}\times G_2\times\dots\times
G_n,H^{(p)})\ar[r]_{\cong}^{\theta_3}\ar[u]^{V_H\circ(-)}&\Mult(G_2\times
\dots\times G_n\times
H^{(p)^*},G_1^{(p)^*})\ar[u]_{F_H^*}\\
\Mult(G_1^{(p)}\times\dots\times
G_n^{(p)},H^{(p)})\ar[r]_{\cong\quad}^{\theta_4\quad}\ar[u]^{(-)\circ\widetilde{F}}&\Mult(G_2^{(p)}\times
\dots\times G_n^{(p)}\times H^{(p)^*},G_1^{(p)^*})\ar[u]_{F^*}}$$
where the horizontal homomorphisms are the isomorphisms given by
Lemma \ref{lem1} (note that $(-)^*=\innHom(-,\BG_m)$) and $F_H^*$
and $F^*$ are respectively the homomorphisms
$(-)\circ(\Id_{G_2}\times\dots\times\Id_{G_n}\times F_{H^*})$ and
$(-)\circ (F_{G_1}\times\dots\times F_{G_n}\times \Id_{H^{(p)^*}})$.
Using the facts that the isomorphism in Lemma \ref{lem1} is
functorial and under the identification $(-)^{(p)^*}\cong
(-)^{*^{(p)}}$, the dual of the Verschiebung of a commutative group
scheme is the Frobenius of the dual group scheme, we deduce that
this diagram is commutative.

The commutativity of the upper square implies that
$$F_{G_1^*}\circ\theta_1(\phi)=\theta_2(\phi\circ \widetilde{V})\qquad (\bigstar)$$ The commutativity of the two
bottom squares implies that
$$\theta_4(\phi^{(p)})\circ
(F_{G_2}\times\dots\times F_{G_n}\times \Id_{H^{(p)^*}})\circ
(\Id_{G_2}\times\dots\times\Id_{G_n}\times
F_{H^*})=\theta_2(V_H\circ \phi^{(p)}\circ \widetilde{F}).$$ The
composition $(F_{G_2}\times\dots\times F_{G_n}\times
\Id_{H^{(p)^*}})\circ (\Id_{G_2}\times\dots\times\Id_{G_n}\times
F_{H^*})$ equals $(F_{G_2}\times\dots\times F_{G_n}\times F_{H^*})$
and one can easily check that the isomorphism $\theta_1$ given in
Lemma \ref{lem1} is compatible with the pullback of the Frobenius,
i.e., $\theta_4(\phi^{(p)})=\theta_1(\phi)^{(p)}$. We have thus
$$\theta_1(\phi)^{(p)}\circ(F_{G_2}\times\dots\times F_{G_n}\times
F_{H^*})=\theta_2(V_H\circ \phi^{(p)}\circ \widetilde{F})\qquad
(\blacktriangle)$$ Writing $\overrightarrow{F}$ for
$(F_{G_2}\times\dots\times F_{G_n}\times F_{H^*})$, we know that
there is a commutative diagram
$$\xymatrix{
G_2\times\dots\times G_n\times H^*\ar[rr]^{\qquad\theta_1(\phi)}\ar[d]_{\overrightarrow{F}}&&G_1^*\ar[d]^{F_{G_1^*}}\\
G_2^{(p)}\times\dots\times G_n^{(p)}\times
H^{*^{(p)}}\ar[rr]_{\qquad\theta_1(\phi)^{(p)}}&&G_1^{*^{(p)}}.}$$
This together with $(\bigstar)$ and $(\blacktriangle)$ imply that
$\theta_2(\phi\circ\widetilde{V})=\theta_2(V_H\circ \phi^{(p)}\circ
\widetilde{F}).$ But $\theta_2$ is injective and therefore
$\phi\circ\widetilde{V}=V_H\circ \phi^{(p)}\circ \widetilde{F}$.

\end{proof}

\begin{rem}
\label{rem14} This lemma is true more generally, i.e., with
$G_i$ and $H$ arbitrary commutative group schemes and not
necessarily finite. But the proof is more complicated and in the
sequel, we will only need the weaker version.\xqed{\lozenge}\end{rem}

Let $G$ be a commutative group scheme over a field $k$ of
characteristic $p$ and $\kappa:G^n\to \Lambda^nG$ the universal
alternating morphism defining $\Lambda^nG$. Then taking the pullback
of $\kappa$ and using the isomorphism $(G^n)^{(p)}\cong(G^{(p)})^n$,
we obtain an alternating morphism
$\kappa^{(p)}:G^{(p)}\to(\Lambda^nG)^{(p)}$. Therefore, there is a
unique homomorphism $\eta:\Lambda^n(G^{(p)})\to (\Lambda^nG)^{(p)}$
such that $\eta\circ \kappa'=\kappa^{(p)}$, where
$\kappa':(G^{(p)})^n\to \Lambda^n(G^{(p)})$ is the universal
alternating morphism of $\Lambda^n(G^{(p)})$.

\begin{lem}
\label{lem14} Let the base field $k$ be perfect of odd
characteristic $p$ and $G$ a commutative group scheme over $k$. Then
the homomorphism
$$\eta:\Lambda^n(G^{(p)})\to(\Lambda^nG)^{(p)}$$
is a natural isomorphism and therefore $(\Lambda^nG)^{(p)}$ together
with the alternating morphism
$\kappa^{(p)}:(G^{(p)})^n\to(\Lambda^nG)^{(p)}$ is an alternating
$n^{\text{th}}$ power of $G^{(p)}$.
\end{lem}

\begin{proof}[\textsc{Proof}]
Note that since the field $k$ is perfect, the functor $(-)^{(p)}$
from the category of affine commutative group schemes over $k$ to
itself is an equivalence of categories. Using the above notation, we
have thus a commutative diagram
$$\xymatrix{
\quad\Hom(\Lambda^nG,H)\quad\ar[rr]^{(-)\circ\kappa}_{\cong}\ar[d]_{(-)^{(p)}}^{\cong}&&\quad\Alt(G^n,H)\ar[d]_{(-)^{(p)}}^{\cong}\quad\\\
\Hom((\Lambda^nG)^{(p)},H^{(p)})\ar[rrd]_{(-)\circ\eta}\ar[rr]^{(-)\circ\kappa^{(p)}}&&\Alt((G^{(p)})^n,H^{(p)})\\
&&\Hom(\Lambda^n(G^{(p)}),H^{(p)})\ar[u]_{\cong}^{(-)\circ\kappa'}.}$$
The above square is commutative because of the functoriality of
$(-)^{(p)}$. It implies that the homomorphism
$$(-)\circ\kappa^{(p)}:\Hom((\Lambda^nG)^{(p)},H^{(p)})\to
\Alt((G^{(p)})^n,H^{(p)})$$ is an isomorphism and so the
homomorphism $(-)\circ\eta$ is also an isomorphism. Since the
functor $(-)^{(p)}$ is an equivalence of categories, we can write
any commutative group scheme as $H^{(p)}$ for some commutative group
scheme $H$. Consequently $\eta$ is an isomorphism.
\end{proof}

\begin{lem}
\label{lem13} Let the base field $k$ be a perfect field of odd
characteristic $p$ and $n$ a positive integer. Then the Verschiebung
$$V:(\Lambda^n\alpha_{p^n})^{(p)}\to\Lambda^n\alpha_{p^n}$$ is
trivial.
\end{lem}

\begin{proof}[\textsc{Proof}]
If we show that every element $\phi$ of $\Alt(\alpha_{p^n}^n,H)$ is
annihilated by the Verschiebung $V_H$ of $H$, i.e., the composite
$$(\alpha_{p^n}^n)^{(p)}\arrover{\phi^{(p)}}H^{(p)}\arrover{V_H}H$$is
zero, then for every element $\psi$ of
$\Hom(\Lambda^n\alpha_{p^n},H)$ we will have $\psi\circ V=0$ and
hence $V=0$ (by putting $H=\Lambda^n\alpha_{p^n}$ and $\psi$ the
identity homomorphism). Indeed, let $\psi:\Lambda^n\alpha_{p^n}\to
H$ be a homomorphism and put $\psi':=\psi\circ\kappa$. Consider the
following commutative diagram
$$\xymatrix{
(\alpha_{p^n}^{(p)})^n\ar[r]^{\kappa^{(p)}}\ar[dr]_{\psi'^{(p)}}&(\Lambda^n\alpha_{p^n})^{(p)}\ar[r]^V\ar[d]^{\psi^{(p)}}
&\Lambda^n\alpha_{p^n}\ar[d]^{\psi}\\
&H^{(p)}\ar[r]_{V_H}&H.}$$ By hypothesis, $V_H\circ\psi'^{(p)}=0$
and therefore $\psi\circ V\circ\kappa^{(p)}=0$. But according to
Lemma \ref{lem14}
$$(-)\circ\kappa^{(p)}:\Hom((\Lambda^nG)^{(p)},H)\to
\Alt((G^{(p)})^n,H)$$ is an isomorphism, which implies that
$\psi\circ V=0$.\\

So we should show that for every $H$, every element $\phi$ of
$\Alt(\alpha_{p^n}^n,H)$ is annihilated by the Verschiebung $V_H$.
We show this in 3 steps.

\begin{itemize}
\item[Step 1)] We show the statement for $H$ finite. According to Lemma \ref{lem12}, we have the following
commutative diagram
$$\xymatrix{
\alpha_{p^n}^{(p)}\times \alpha_{p^n}^{(p)}\times\dots\times
\alpha_{p^n}^{(p)}\ar[rr]^{\qquad\quad\phi^{(p)}}&&H^{(p)}\ar[dd]^{V_H}\\
\alpha_{p^n}^{(p)}\times \alpha_{p^n}\times\dots\times
\alpha_{p^n}\ar[d]_{\widetilde{V}}\ar[u]^{\widetilde{F}}&&\\
\alpha_{p^n}\times \alpha_{p^n}\times\dots\times
\alpha_{p^n}\ar[rr]_{\qquad\phi}&&H}$$ where
$\widetilde{V}=V_{\alpha_{p^n}}\times
\Id_{\alpha_{p^n}}\times\dots\times\Id_{\alpha_{p^n}}$ and
$\widetilde{F}=\Id_{\alpha_{p^n}^{(p)}}\times F_{\alpha_{p^n}}\times
\dots\times F_{\alpha_{p^n}}$. But the Verschiebung is trivial on
$\alpha_{p^n}$, so the composite $\phi\circ \widetilde{V}$ is
trivial, because $\phi$ is multilinear, and hence
$V_H\circ\phi^{(p)}\circ \widetilde{F}=0$. We want to show that this
implies $V_H\circ\phi^{(p)}$ is zero.

We know that we can write the Frobenius $F:\alpha_{p^n}\to
\alpha_{p^n}^{(p)}$ as the composite
$$\alpha_{p^n}\ovset{q}{\longonto} \alpha_{p^{n-1}}\ovset{i}{\into} \alpha_{p^n}\unset{\cong}{\arrover{\theta}} \alpha_{p^n}^{(p)}$$
where the epimorphism and the monomorphism are the natural ones and
the isomorphism $\theta:\alpha_{p^n}\cong \alpha_{p^n}^{(p)}$ is
given by the Hopf algebra isomorphism
$$k[X]/(X^{p^n})\otimes_{k,\sigma}k\to k[X]/(X^{p^n}),\quad
\overline{x}\otimes a\mapsto a\cdot\overline{x},\ \ \forall a\in
k.$$ We can thus write $\widetilde{F}$ as the composition
$$\alpha_{p^n}^{(p)}\times \alpha_{p^n}\times\dots\times \alpha_{p^n}\ovset{\widetilde{q}}{\to}
\alpha_{p^n}\times \alpha_{p^{n-1}}\times\dots\times\alpha_{p^{n-1}}
\ovset{\widetilde{i}}{\to}\alpha_{p^n}^n\ovset{\theta^n}{\to}(\alpha_{p^n}^{(p)})^n
$$
where $\widetilde{q}$ is $\theta^{-1}\times q\times \dots\times q$
and $\widetilde{i}$ is the restriction map $\Id_{\alpha_{p^n}}\times
i\times\dots\times i$. Since $\widetilde{q}$ is epimorphic and the
composition $V_H\circ\phi^{(p)}\circ \widetilde{F}$ is zero, we have
that $V_H\circ\phi^{(p)}\circ  \theta^n\circ \widetilde{i}=0$. Since
$\phi$ is alternating, $\phi^{(p)}$ is alternating too. Therefore
the morphism $V_H\circ\phi^{(p)}\circ  \theta^n$ is alternating. It
has a trivial restriction to $\alpha_{p^n}\times
\alpha_{p^{n-1}}\times\dots\times\alpha_{p^{n-1}}$ and we have a
short exact sequence $0\to\alpha_{p^{n-1}}\to \alpha_{p^n}\to
\alpha_p\to0$. We can thus apply Theorem \ref{lem7} (a) and conclude
that the morphism $V_H\circ\phi^{(p)}\circ \theta^n$ is zero as
well. Since $\theta^n$ is an isomorphism, the morphism
$V_H\circ\phi^{(p)}$ is zero.

\item[Step 2)] We show the statement with $H$ of finite type.
According to Proposition 2.3 in \cite{Pink1}, the morphism $\phi$
factors through a finite subgroup scheme $H'$ of $H$, i.e., the
following diagram is commutative
$$\xymatrix{
\alpha_{p^n}^n\ar[rr]^{\phi}\ar[dr]^{\phi'}&&H\\
&H'.\ar@{^{ (}->}[ur]&}$$ We have thus a commutative diagram
$$\xymatrix{
(\alpha_{p^n}^{(p)})^n\ar[r]^{\phi^{(p)}}\ar[dr]_{\phi'^{(p)}}&H^{(p)}\ar[r]^{V_H}&H\\
&H'^{(p)}\ar@{^{ (}->}[u]\ar[r]_{V_{H'}}&H'.\ar@{^{ (}->}[u]}$$ By
step 1, we have $V_{H'}\circ\phi'^{(p)}=0$. Hence
$V_H\circ\phi^{(p)}=0$.

\item[Step 3)] Now we show the statement for general $H$. We know
that we can write $H=\invlim H_i$ with commutative schemes $H_i$ of
finite type. Let $\lambda_i:H\to H_i$ be the canonical homomorphisms
of the inverse limit and put $\phi_i:=\lambda_i\circ\phi$. For ever
$i$ we have a commutative diagram
$$\xymatrix{
(\alpha_{p^n}^{(p)})^n\ar[r]^{\phi^{(p)}}\ar[dr]_{\phi_i^{(p)}}&H^{(p)}\ar[r]^{V_H}\ar[d]^{\lambda_i^{(p)}}&H\ar[d]^{\lambda_i}\\
&H_i\ar[r]_{V_{H_i}}&H_i.}$$ By Step 2, the composition
$V_{H_i}\circ\phi_i^{(p)}$ is trivial and thus we have for all $i$
that the composition $\lambda_i\circ V_H\circ \phi^{(p)}=0$. Since
$H=\invlim H_i$, we conclude that $V_H\circ\phi^{(p)}=0$.
\end{itemize}
\end{proof}

\begin{prop}
\label{prop12} Let the base scheme be $\Spec k$ for a perfect field
of odd characteristic $p$ and $n$ a positive integer. Then there is
an isomorphism
$$\Lambda^n\alpha_{p^n}\cong\alpha_p.$$
\end{prop}

\begin{proof}[\textsc{Proof}]
If $n=1$ then this is a tautology, so assume $n>1$. We know by
Proposition \ref{ex5} that $\innAlt(\alpha_{p^n}^n,\BG_a)\cong\BG_a$ and
therefore, $\Hom(\Lambda^n\alpha_{p^n},\BG_a)\cong\BG_a(k)=k$. This
shows that the group scheme $\Lambda^n\alpha_{p^n}$ is not trivial.
Assume that we have an epimorphism
$\pi:\alpha_p\onto\Lambda^n\alpha_{p^n}$. This implies that $\pi$ is
in fact an isomorphism, because its kernel could not be the whole
group scheme $\alpha_p$ (since otherwise the image,
$\Lambda^n\alpha_{p^n}$, would be trivial) and since $\alpha_p$ is
simple, the kernel should be zero. Consequently, $\pi$ is a
monomorphism too and hence an isomorphism. It is thus sufficient to
show that such an epimorphism exists.

We know from Proposition \ref{prop7} that there is an epimorphism
$\theta:\alpha_p^{\otimes n}\onto\Lambda^n\alpha_{p^n}$. Consider
the following commutative diagram
$$\xymatrix{
\alpha_p^{\otimes
n^{(p)}}\ar@{->>}[r]^{\theta^{(p)}}\ar[d]_{V}&\Lambda^n\alpha_{p^n}^{(p)}\ar[d]^{V'}\\
\alpha_p^{\otimes
n}\ar@{->>}[r]^{\theta}\ar@{->>}[d]&\Lambda^n\alpha_{p^n}\ar[d]\\
C\ar@{->>}[r]^{\overline{\theta}}&C'}$$ where $V$ and $V'$ are
Verschiebung and $C$ and $C'$ are the cokernels of $V$ and $V'$. By
Lemma \ref{lem11}, $C$ is isomorphic to $\alpha_p$. We know from
Lemma \ref{lem13} that the image of $V'$ is zero and hence its
cokernel $C'$ is isomorphic to $\Lambda^n\alpha_{p^n}$. Since
$C\arrover{\overline{\theta}}C'$ is epimorphic, we get the desired
epimorphism $\alpha_p\onto \Lambda^n\alpha_{p^n}$.
\end{proof}

\phantomsection

\end{document}